\title{\vspace{-0.7cm}Density theorems for bipartite graphs and related Ramsey-type results}
\author{
Jacob Fox\thanks{Department of Mathematics, Princeton, Princeton,
NJ. Email: {\tt jacobfox@math.princeton.edu}. Research supported by
an NSF Graduate Research Fellowship and a Princeton Centennial
Fellowship.} \and Benny Sudakov\thanks{Department of Mathematics,
Princeton, Princeton, NJ. Email: {\tt bsudakov@math.princeton.edu}.
Research supported in part by NSF CAREER award DMS-0546523, NSF
grant DMS-0355497 and by a USA-Israeli BSF grant.}}

\documentclass[11pt]{article}
\usepackage{amsfonts,amssymb,amsmath,latexsym}

\newenvironment{proof}
      {\medskip\noindent{\bf Proof.}\hspace{1mm}}
      {\hfill$\Box$\medskip}

\def\qed{\ifvmode\mbox{ }\else\unskip\fi\hskip 1em plus 10fill$\Box$}

\newtheorem{theorem}{Theorem}[section]

\newtheorem{lemma}[theorem]{Lemma}

\newtheorem{corollary}[theorem]{Corollary}

\newtheorem{definition}[theorem]{Definition}

\begin{document}
\date{}

\maketitle

\begin{abstract}
In this paper, we present several density-type theorems which show
how to find a copy of a sparse bipartite graph in a graph of
positive density. Our results imply several new bounds for classical
problems in graph Ramsey theory and improve and generalize earlier
results of various researchers. The proofs combine probabilistic
arguments with some combinatorial ideas. In addition, these
techniques can be used to study properties of graphs with a
forbidden induced subgraph, edge intersection patterns in
topological graphs, and to obtain several other Ramsey-type
statements.
\end{abstract}

\section{Background and Introduction}

Ramsey theory refers to a large body of deep results in mathematics
whose underlying philosophy is captured succinctly by the statement
that ``In a large system, complete disorder is impossible.'' This is
an area in which a great variety of techniques from many branches of
mathematics are used and whose results are important not only to
graph theory and combinatorics but also to logic, analysis, number
theory, and geometry. Since the publication of the seminal paper of
Ramsey in 1930, this subject has grown with increasing vitality, and
is currently among the most active areas in combinatorics.

For a graph $H$, the {\it Ramsey number} $r(H)$ is the least
positive integer $n$ such that every two-coloring of the edges of
complete graph $K_n$ on $n$ vertices, contains a monochromatic copy
of $H$. Ramsey's theorem states that $r(H)$ exists for every graph
$H$. A classical result of Erd\H{o}s and Szekeres~\cite{ErSz}, which
is a quantitative version of Ramsey's theorem, implies that $r(K_k)
\leq 2^{2k}$ for every positive integer $k$. Erd\H{o}s~\cite{Er3}
showed using probabilistic arguments that $r(K_k) > 2^{k/2}$ for $k
> 2$. Over the last sixty years, there has been several
improvements on these bounds (see, e.g., \cite{Co}). However,
despite efforts by various researchers, the constant factors in the
above exponents remain the same.

Determining or estimating Ramsey numbers is one of the central
problem in combinatorics, see the book {\it Ramsey theory}
\cite{GrRoSp} for details. Besides the complete graph, the next most
classical topic in this area concerns the Ramsey numbers of sparse
graphs, i.e., graphs with certain upper bound constraints on the
degrees of the vertices. The study of these Ramsey numbers was
initiated by Burr and Erd\H{o}s in 1975, and this topic has since
placed a  central role in graph Ramsey theory.

An {\it induced subgraph} is a subset of the vertices of a graph
together with all edges whose both endpoints are in this subset.
There are several results and conjectures which indicate that graphs
which do not contain a fixed induced subgraph are highly structured.
In particular, the most famous conjecture of this sort by Erd\H{o}s
and Hajnal \cite{ErHa} says that every graph $G$ on $n$ vertices
which does not contain a fixed induced subgraph $H$ has a clique or
independent set of size a power of $n$. This is in striking contrast
with the general case where one can not guarantee a clique or
independent set of size larger than logarithmic in the number of
vertices.

Results in Ramsey theory generally say that if a large enough
structure is partitioned into a small number of parts, then one of
the resulting parts will contain some desired substructure.
Sometimes, a stronger {\it density-type} result can be proved, which
shows that any dense subset of a large enough structure contains the
desired substructure. One famous example is Szemer\'edi's theorem,
which says that every subset of the positive integers of positive
upper density contains arbitrarily long arithmetic progressions. It
strengthens the earlier result of van der Waerden that every finite
partition of the positive integers contain arbitrarily long
arithmetic progressions, and has led to many deep and beautiful
results in various areas of mathematics, including the recent
spectacular result of Green and Tao that there are arbitrarily long
arithmetic progressions in primes.

It is easy to see that Ramsey's theorem has no density-type
analogue. Indeed, the complete bipartite graph with both parts of
size $n/2$ has $n^2/4$ edges, i.e., more than half the total
possible number of edges, and still does not contain a triangle.
However, for bipartite graphs, a density version exists as was shown
by K\"ovari, S\'os, and Tur\'an \cite{KoSoTu} in 1954.

In this paper, we present several density-type theorems which show
how to find a copy of a sparse bipartite graph in a graph of
positive density. Our results imply several new bounds for classical
problems in graph Ramsey theory and improve and generalize earlier
results of various researchers. The proofs combine probabilistic
arguments with some combinatorial ideas. In addition, these
techniques can be used to study edge intersection patterns in
topological graphs, make some progress towards the Erd\H{o}s-Hajnal
conjecture, and obtain several other Ramsey-type statements. In the
subsequent sections we present in full detail our theorems and
compare them with previously obtained results.

\subsection{Ramsey numbers and density-type theorems for bipartite graphs}\label{sub:Ramsey}

Estimating Ramsey numbers is one of the central (and difficult)
problems in modern combinatorics. Among the most interesting
questions in this area are the linear bounds for Ramsey numbers of
graphs with certain degree constraints. In 1975, Burr and Erd\H{o}s
\cite{BuEr} conjectured that, for each positive integer $\Delta$,
there is a constant $c(\Delta)$ such that every graph $H$ with $n$
vertices and maximum degree $\Delta$ satisfies $r(H) \leq
c(\Delta)n$. This conjecture was proved by Chvat\'al, R\"odl,
Szemer\'edi, and Trotter \cite{ChRoSzTr}. Their proof is a beautiful
illustration of the power of Szemer\'edi's regularity lemma
\cite{KoSi}. However, the use of this lemma makes an upper bound on
$c(\Delta)$ to grow as a tower of $2$s with height polynomial in
$\Delta$. Since then, the problem of determining the correct order
of magnitude of $c(\Delta)$ as a function of $\Delta$ has received
considerable attention from various researchers. Still using a
variant of the regularity lemma, Eaton \cite{Ea} showed that
$c(\Delta)< 2^{2^{c\Delta}}$ for some fixed $c$. A novel approach of
Graham, R\"odl, and Rucinski \cite{GrRoRu} gave the first linear
upper bound on Ramsey numbers of bounded degree graphs without using
any form of the regularity lemma. Their proof implies that
$c(\Delta)<2^{c\Delta\log^2 \Delta}$. (Here, and throughout the
paper, all logarithms are base $2$.)

The case of bipartite graphs with bounded degree was studied by
Graham, R\"odl, and Rucinski more thoroughly in \cite{GrRoRu1},
where they improved their upper bound, showing that $r(H) \leq
2^{\Delta\log \Delta +O(\Delta)}n$ for every bipartite graph $H$
with $n$ vertices and maximum degree $\Delta$. As they point out,
their proof does not give a stronger density-type result. In the
other direction, they proved that there is a positive constant $c$
such that, for every $\Delta \geq 2$ and $n \geq \Delta+1$, there is
a bipartite graph $H$ with $n$ vertices and maximum degree $\Delta$
satisfying $r(H) \geq 2^{c\Delta}n$. Closing the gaps between these
two bounds remained a challenging open problem. In this paper, we
solve this problem by showing that the correct order of magnitude of
the Ramsey number of bounded degree bipartite graphs is essentially
given by the lower bound. This follows from the following
density-type theorem.

\begin{theorem}\label{main}
Let $H$ be a bipartite graph with $n$ vertices and maximum degree
$\Delta \geq 1$. If $\epsilon>0$ and $G$ is a graph with $N \geq
32\Delta \epsilon^{-\Delta}n$ vertices and at least $\epsilon{N
\choose 2}$ edges, then $H$ is a subgraph of $G$.
\end{theorem}

Taking $\epsilon=1/2$ together with the majority color in a
$2$-coloring of the edges of $K_N$, we obtain a corollary which
gives a best possible upper bound up to the constant factor in the
exponent on Ramsey numbers of bounded degree bipartite graphs.

\begin{corollary}\label{cormain}
If $H$ is bipartite, has $n$ vertices and maximum degree $\Delta
\geq 1$, then $r(H) \leq \Delta2^{\Delta+5}n$.
\end{corollary}

Moreover, the above theorem also easily gives an upper bound on
multicolor Ramsey numbers of bipartite graphs. The $k$-color Ramsey
number $r(H_1,\ldots,H_k)$ is the least positive integer $N$ such
that for every $k$-coloring of the edges of the complete graph
$K_N$, there is a monochromatic copy of $H_i$ in color $i$ for some
$1 \leq i \leq k$. Taking $\epsilon=1/k$ in Theorem \ref{main} and
considering the majority color in a $k$-coloring of the edges of a
complete graph shows that for bipartite graphs $H_1,\ldots,H_k$ each
with $n$ vertices and maximum degree at most $\Delta$,
$r(H_1,\ldots,H_k) \leq 32\Delta k^{\Delta}n$.

One family of bipartite graphs that have received particular
attention are the $d$-cubes. The $d$-cube $Q_d$ is the $d$-regular
graph with $2^{d}$ vertices whose vertex set is $\{0,1\}^d$ and two
vertices are adjacent if they differ in exactly one coordinate. Burr
and Erd\H{o}s conjectured that $r(Q_d)$ is linear in the number of
vertices of the $d$-cube. Beck \cite{Be} proved that $r(Q_d) \leq
2^{cd^2}$. The bound of Graham et al. \cite{GrRoRu} gives the
improvement $r(Q_d) \leq 8(16d)^{d}$. Shi \cite{Sh}, using ideas of
Kostochka and R\"odl \cite{KoRo}, proved that $r(Q_d) \leq
2^{(\frac{3+\sqrt{5}}{2})d+o(d)}$, which is a polynomial bound in
the number of vertices with exponent $\frac{3+\sqrt{5}}{2} \approx
2.618$. A very special case of Corollary \ref{cormain}, when
$H=Q_d$, gives immediately the following improved result.

\begin{corollary}\label{cubecor}
For every positive integer $d$, $r(Q_d) \leq d2^{2d+5}$.
\end{corollary}

A graph is {\it $d$-degenerate} if every subgraph of it has a vertex
of degree at most $d$. Notice that graphs with maximum degree $d$
are $d$-degenerate. This notion nicely captures the concept of
sparse graphs as every $t$-vertex subgraph of a $d$-degenerate graph
has at most $td$ edges. (Indeed, remove from the subgraph a vertex
of minimum degree, and repeat this process in the remaining
subgraph.) Burr and Erd\H{o}s \cite{BuEr} conjectured that, for each
positive integer $d$, there is a constant $c(d)$ such that $r(H)
\leq c(d)n$ for every $d$-degenerate graph $H$ on $n$ vertices. This
well-known and difficult conjecture is a substantial generalization
of the above mentioned results on Ramsey numbers of bounded degree
graphs and progress on this problem was made only recently.

Kostochka and R\"odl \cite{KoRo1} were the first to prove a
polynomial upper bound on the Ramsey numbers of $d$-degenerate
graphs. They showed that $r(H) \leq c_dn^2$ for every $d$-degenerate
graph $H$ with $n$ vertices. A nearly linear bound of the form $r(H)
\leq c_dn^{1+\epsilon}$ for any fixed $\epsilon>0$ was obtained in
\cite{KoSu}. For bipartite $H$, Kostochka and R\"odl proved that
$r(H) \leq d^{d+o(d)}\Delta n$, where $\Delta$ is the maximum degree
of $H$. Kostochka and Sudakov \cite{KoSu} proved that $r(H) \leq
2^{O(\log^{2/3}n)}n$ for every $d$-degenerate bipartite graph $H$
with $n$ vertices and constant $d$. Here we improve on both of these
results.

\begin{theorem}\label{degenerate}
If $d/n \leq \delta \leq 1$, $H$ is a $d$-degenerate bipartite graph
with $n$ vertices and maximum degree $\Delta \geq 1$, $G$ is a graph
with $N$ vertices and at least $\epsilon {N \choose 2}$ edges, and
$N \geq 2^{12}\epsilon^{-(1/\delta +3)d-2}\Delta^{\delta}n$, then
$H$ is a subgraph of $G$.
\end{theorem}

For $\delta$ and $H$ as in the above theorem, taking $\epsilon=1/2$
and considering the majority color in a $2$-coloring of the edges of
$K_N$ shows that
$$r(H) \leq 2^{\delta^{-1}d+3d+14}\Delta^{\delta}n.$$ This new upper
bound on Ramsey numbers for bipartite graphs is quite versatile.
Taking $\delta=1$, we have $r(H) \leq 2^{4d+14}\Delta n$ for
bipartite $d$-degenerate graphs with $n$ vertices and maximum degree
$\Delta$. This improves upon the bound of Kostochka and R\"odl. If
$\Delta \geq 2^{d}$, then taking $\delta=(\frac{d}{\log
\Delta})^{1/2}$, we have $$r(H) \leq 2^{2\sqrt{d\log
\Delta}+3d+14}n$$ for bipartite $d$-degenerate graphs $H$ with $n$
vertices and maximum degree $\Delta$. In particular, we have $r(H)
\leq 2^{O(\log^{1/2} n)}n$ for constant $d$. This improves on the
bound of Kostochka and Sudakov, and is another step closer to the
Burr-Erd\H{o}s conjecture.

Moreover, as long as $\Delta$ is at most exponential in $d$, we
still have $r(H) \leq 2^{O(d)}n$. This has interesting applications
to another notion of sparseness introduced by Chen and Schelp
\cite{ChSc}. A graph is {\it $p$-arrangeable} if there is an
ordering $v_1,\ldots,v_n$ of the vertices such that for any vertex
$v_i$, its neighbors to the right of $v_i$ have together at most $p$
neighbors to the left of $v_i$ (including $v_i$). This is an
intermediate notion of sparseness not as strict as bounded degree
though not as general as bounded degeneracy. Extending the result of
\cite{ChRoSzTr}, Chen and Schelp proved that there is a constant
$c(p)$ such that every $p$-arrangeable graph $H$ on $n$ vertices has
Ramsey number at most $c(p)n$. This gives linear Ramsey numbers for
planar graphs and more generally for graphs that can be drawn on a
bounded genus surfaces. The best known bound \cite{GrRoRu} for
$p$-arrangeable bipartite $H$ is $r(H) \leq 2^{cp\log p}n$, where
$c$ is a constant. The proof of Theorem \ref{degenerate} can be
modified to give $r(H) \leq 2^{cp}n$ for every $p$-arrangeable
bipartite graph $H$, which is an essentially best possible bound.
Note that for every vertex $v_i$ in a $p$-arrangeable graph, there
is a subset $S_i \subset \{v_1,\ldots,v_{i-1}\}$ of size at most
$p-1$ such that for any vertex $v_j,j>i$ adjacent to $v_i$, its
neighbors in $\{v_1,\ldots,v_{i-1}\}$ form a subset of $S_i$.
Therefore, there are at most $2^{p-1}$ distinct such subsets of
neighbors. This important observation essentially allows us to treat
$p$-arrangeable bipartite graphs as if they were $p$-degenerate
graphs with maximum degree at most $2^{p-1}$, which in turn gives
the above bound on Ramsey numbers.

In spite of the above mentioned progress, the Burr-Erd\H{o}s
conjecture is still open even for the special case of $d$-degenerate
bipartite graphs in which every vertex in one part has degree at
most $d \geq 3$. Using our approach, one can make some progress on
this special case, which is discussed in the concluding remarks.

It seems plausible that $r(H) \leq 2^{c\Delta}n$ holds in general
for every graph $H$ with $n$ vertices and maximum degree $\Delta$.
The following result shows that this is at least true for graphs of
bounded chromatic number.

\begin{theorem}\label{chromatic}
If $H$ has $n$ vertices, chromatic number $q$, and maximum degree
$\Delta$, then $r(H) \leq 2^{4q\Delta}n$.
\end{theorem}

\subsection{Subgraph Multiplicity}\label{multiplicity}
Recall that Ramsey's theorem states that every $2$-edge-coloring of
a sufficiently large complete graph $K_N$ contains at least one
monochromatic copy of a given graph $H$. Let $c_{H,N}$ denote the
fraction of copies of $H$ in $K_N$ that must be monochromatic in any
$2$-edge-coloring. By an averaging argument, $c_{H,N}$ is a bounded,
monotone increasing function in $N$, and therefore has a limit $c_H$
as $N \rightarrow \infty$. The constant $c_H$ is known as the {\it
Ramsey multiplicity constant} for the graph $H$. It is simple to
show for $H$ with $m$ edges that $c_H \leq 2^{1-m}$, where this
bound comes from considering a random $2$-edge-coloring of $K_N$
with each coloring equally likely.

Erd\H{o}s and in a more general form Burr and Rosta suggested that
the Ramsey multiplicity constant is achieved by a random coloring.
These conjectures are false as was demonstrated Thomason \cite{Th1}
even for $H$ being any complete graph $K_n$ with $n \geq 4$.
Moreover, as shown in \cite{Fo}, there are $H$ with $m$ edges and
$c_H \leq m^{-m/2+o(m)}$, which demonstrates that the random
coloring is far from being optimal for some graphs.

For bipartite graphs the situation seems to be very different. The
{\it edge density} of a graph is the fraction of pairs of vertices
that are edges. The conjectures of Simonovits \cite{Sim} and
Sidorenko \cite{Si3} suggest that for any bipartite $H$ the number
of its copies in any graph $G$ on $N$ vertices and edge density
$\epsilon$ ($\epsilon > N^{-\gamma(H)}$) is asymptotically at least
the same as in the $N$-vertex random graph with edge density
$\epsilon$. So far it is known only in very special cases, i.e., for
complete bipartite graphs, trees, even cycles (see \cite{Si3}), and
recently for cubes \cite{Ha}. Our Theorem \ref{main} can be
strengthened as follows to give additional evidence for the validity
of this conjecture.

\begin{theorem}\label{main1}
Let $H$ be a bipartite graph with $n$ vertices and maximum degree $d
\geq 1$. If $\epsilon>0$ and $G$ is a graph with $N \geq 32d
\epsilon^{-d}n$ vertices and at least $\epsilon{N \choose 2}$ edges,
then $G$ contains at least $(2^7d)^{-n/2}\epsilon^{dn}N^n$ labeled
copies of $H$.
\end{theorem}

Notice that this theorem roughly says that a large graph with edge
density $\epsilon$ contains at least $\epsilon^{dn}$ fraction of all
possible copies of $H$. If $H$ is $d$-regular, i.e., has $dn/2$
edges, then the random graph with edge density $\epsilon$ contains
$\epsilon^{dn/2}$ fraction of all possible copies of $H$. This shows
that for regular $H$ the exponent of $\epsilon$ in the above theorem
is only by a factor $2$ away from the conjectured bound. Moreover,
the same is true with a different factor for every $d$-degenerate
bipartite graph $H$ with maximum degree at most exponential in $d$.
This follows from an extension of our result on $d$-degenerate
bipartite graphs which is discussed in Section 3. A similar
extension for graphs with bounded chromatic number is discussed in
Section 4.

\subsection{Subdivided subgraphs in dense graphs}\label{sub:subdivided}

A {\it topological copy} of a graph $H$ is any graph formed by
replacing edges of $H$ by internally vertex disjoint paths. This is
an important notion in graph theory, e.g., the celebrated theorem of
Kuratowski uses it to characterize planar graphs.  In the special
case in which each of the paths replacing edges of $H$ has length
$t+1$, we obtain a {\it $t$-subdivision} of $H$. An old conjecture
of Mader and Erd\H{o}s-Hajnal which was proved in \cite{BoTh,KoSz}
says that there is a constant $c$ such that every graph with $n$
vertices and at least $cp^2n$ edges contains a topological copy of
$K_p$.

Erd\H{o}s \cite{Er} asked whether every graph on $n$ vertices with
$c_1n^2$ edges contains a $1$-subdivision of a complete graph $K_m$
with $m \geq c_2 \sqrt{n}$ for some constant $c_2$ depending on
$c_1$. Note that the above mentioned result implies that any such
graph on $n$ vertices will contain a topological copy of a complete
graph on $\Omega(\sqrt{n})$ vertices, but not necessarily a
$1$-subdivision. The existence of such a subdivision was proved in
\cite{AlKrSu}, giving a positive answer to the question of
Erd\H{o}s. Note that clique of order $O(\sqrt{n})$ has $O(n)$ edges.
So it is natural to ask whether the conjecture of Erd\H{o}s can be
generalized to show that under the same conditions as above one can
find a $1$-subdivision of every graph with $O(n)$ edges, not just of
a clique.

A result closely related to this question was obtained by Alon et
al.~in \cite{AlDuLeRoYu} (see also \cite{KoSi}). They proved, using
Szemeredi's regularity lemma, that any graph with $n$ vertices and
at least $c_1n^2$ edges contains a topological copy of every graph
with at most $c_2n$ edges ($c_2$ depends on $c_1$). Moreover, their
proof shows that the topological copy of $H$ can be taken to be a
$3$-subdivision of $H$.

Motivated by the conjecture of Burr and Erd\H{o}s that graphs with
bounded degeneracy have linear Ramsey numbers, Alon \cite{Al1}
proved that any graph on $n$ vertices in which no two vertices of
degree at least three are adjacent has Ramsey number at most $12n$.
In particular, the Ramsey number of a $1$-subdivision of an
arbitrary
 graph with $n$ edges is linear in $n$.

The following density-type theorem improves on these previous
results concerning subdivided graphs, and gives a positive answer to
the generalization of the Erd\H{o}s conjecture mentioned above.

\begin{theorem}\label{subdivided1}
Let $H$ be a graph with $n$ edges and no isolated vertices and let
$G$ be a graph with $N$ vertices and $\epsilon N^2$ edges such that
$N \geq 100\epsilon^{-3}n$. Then $G$ contains the $1$-subdivision of
$H$.
\end{theorem}

\subsection{Forbidden induced subgraphs}\label{sub:ErdosHajnal}

A graph is {\it $H$-free} if it does not contain $H$ as an {\em
induced} subgraph. A basic property of large random graphs is that
they almost surely contain any fixed graph $H$ as an induced
subgraph. Therefore, there is a general belief that $H$-free graphs
are highly structured. For example, Erd\H{o}s and Hajnal \cite{ErHa}
proved that every $H$-free graph on $N$ vertices contains a clique
or independent set of size at least $2^{c\sqrt{\log N}}$, where
$c>0$ only depends on $H$. This is in striking contrast with the
general case where one can not guarantee a clique or independent set
of size larger than logarithmic in $N$. Erd\H{o}s-Hajnal further
conjectured that this bound can be improved to $N^{c}$. This famous
conjecture has only been solved for some particular $H$ (see, e.g,
\cite{AlPaSo} and \cite{ChSa}).

An interesting partial result for the general case was obtained by
Erd\H{o}s, Hajnal, and Pach \cite{ErHaPa}. They show that every
$H$-free graph $G$ with $N$ vertices or its complement $\bar G$
contains a complete bipartite graph with parts of size $N^{c(H)}$.
We obtain a strengthening of this result which brings it closer to
the Erd\H{o}s-Hajnal conjecture.

\begin{theorem}\label{halferdoshajnal}
For every graph $H$, there is $c>0$ such that any $H$-free graph on
$N$ vertices contains a complete bipartite graph with parts of size
$N^{c}$ or an independent set of size $N^{c}$.
\end{theorem}

To get a better understanding of the properties of $H$-free graphs,
one can naturally ask for an {\it asymmetric} version of the
Erd\H{o}s-Hajnal result. The proof in \cite{ErHa} first shows that
every $H$-free graph $G$ on $N$ vertices contains a {\it perfect}
induced subgraph of order $2^{c\sqrt{\log N}}$. It then uses a well
known fact that every perfect graph on $n$ vertices contains a
clique or an independent set of order $\sqrt{n}$. Therefore, it is
not clear how to adjust this proof to improve the bound of
$2^{c\sqrt{\log N}}$ in the case when we know that the maximum
clique or independent set in $G$ is rather small. The general
framework we develop in this paper can be used to obtain such a
generalization of the Erd\H{o}s-Hajnal result.

\begin{theorem}\label{offdiagonal}
There exists $c=c(H)>0$ such that for any $H$-free graph $G$ on $N$
vertices and $n_1,n_2$ satisfying $(\log n_1)(\log n_2) \leq c\log
N$, $G$ contains a clique of size $n_1$ or an independent set of
size $n_2$.
\end{theorem}

\subsection{Edge intersection patterns in topological graphs}
\label{topsubsection}

The origins of graph theory are closely connected with topology and
geometry. Indeed, the first monograph on graph theory, by K\"onig
 in 1935, was entitled {\it Combinatorial Topology of
Systems of Segments}. In recent years, geometric graph theory, which
studies intersection patterns of geometric objects and graph
drawings, has rapidly developed.

A {\it topological graph} is a graph drawn in the plane with
vertices as points and edges as curves connecting endpoints and
passing through no other vertices. A topological graph is {\it
simple} if any two edges have at most one point in common. A very
special case of simple topological graphs is {\it geometric graphs}
in which edges are straight-line segments. There are many well known
open problems about graph drawings and in particular edge
intersection patterns of topological graphs. Even some innocent
looking questions in this area can be quite difficult.

For example, more than 40 years ago Conway asked what is the maximum
size of a {\it thrackle}, that is, a simple topological graph in
which every two edges intersect. He conjectured that every
$n$-vertex thrackle has at most $n$ edges. Lov\'asz, Pach, and
Szegedy \cite{LoPaSz} were the first to prove a linear upper bound
on the number of edges in a thrackle, and despite some improvement
in \cite{CaNi}, the conjecture is still open. On the other hand,
Pach and T\'oth \cite{PaTo} constructed drawings of the complete
graph in the plane with each pair of edges having at least one and
at most two points in common. Hence, to ensure a pair of disjoint
edges, the assumption that the topological graph is simple is
necessary.

For dense simple topological graphs, one might expect to
obtain a much stronger conclusion than that of Conway's conjecture,
showing that these graphs contain large patterns of pairwise disjoint
edges. Our next theorem proves that this is indeed true, extending an earlier result
of Pach and Solymosi \cite{PaSo} for geometric graphs.

\begin{theorem}\label{bipartitedisjoint}
For each $\gamma>0$ there is $\delta>0$ and $n_0$ such that every
simple topological graph $G=(V,E)$ with $n \geq n_0$ vertices and $m
\geq \gamma n^2$ edges contains two disjoint edge subsets $E_1,E_2$
each of cardinality at least $\delta n^2$ such that every edge in
$E_1$ is disjoint from every edge in $E_2$.
\end{theorem}

This result has a natural interpretation in the context of Ramsey
theory for intersection graphs. The {\it intersection graph} of a
collection of curves in the plane has a vertex for each curve and
two of its vertices are adjacent if their corresponding curves
intersect. It is easy to show that the $1$-subdivision of $K_5$ is
not an intersection graph of curves in the plane and thus the edge
intersection graph of a topological graph has a fixed forbidden
induced subgraph. Therefore, the properties of intersection graphs
are closely related to the Erd\H{o}s-Hajnal conjecture mentioned in
the previous subsection, and one might expect to find in these
graphs two large vertex subsets with no edges between them.
Nevertheless, Theorem \ref{bipartitedisjoint} is still quite
surprising because it shows that the edge intersection graph of any
dense simple topological graph contains two {\it linear}-sized
subsets with no edges between them.

Another interesting Ramsey-type problem is to estimate the maximum
number of pairwise disjoint edges in any complete simple topological
graph. Pach and T\'oth \cite{PaTo} proved that every simple
topological graph of order $n$ without $k$ pairwise disjoint edges
has $O\big(n(\log n)^{4k-8}\big)$ edges. They use this to show that
every complete simple topological graph of order $n$ has
$\Omega(\log n/\log \log n)$ pairwise disjoint edges. Using Theorem
\ref{bipartitedisjoint}, we give a modest improvement on this bound
(the truth here is probably $n^{\epsilon}$). Our result is valid for
dense (not only complete) simple topological graphs as well.

\begin{corollary}\label{cor:drawing}
There is $\epsilon>0$ such that every complete simple topological
graph of order $n$ contains $\Omega\big((\log n)^{1+\epsilon}\big)$
pairwise disjoint edges.
\end{corollary}

The proof of the above two results rely on a new theorem concerning
the edge distribution of $H$-free graphs. It extends earlier results
of \cite{Ro} and \cite{FoSu} which show that $H$-free graphs contain
large induced subgraphs that are very sparse or dense. However,
these results are not sufficient for our purposes. We prove that
$H$-free graphs satisfying a seemingly weak edge density condition
contain a very dense linear-sized induced subgraph.

\subsection{Induced Ramsey numbers}\label{sub:inducedRamsey}
In the early 1970's an important generalization of Ramsey's theorem,
the Induced Ramsey Theorem, was discovered independently by Deuber;
Erd\H{o}s, Hajnal, and Posa; and R\"odl. We write $$G
\xrightarrow{\textrm{ind}}(H_1,\ldots,H_k)$$ if, for every
$k$-coloring of the edges of $G$ with colors $1,\ldots,k$, there is
an index $i$ and an induced copy of $H_i$ in $G$ that is
monochromatic of color $i$. The Induced Ramsey Theorem states that
for all graphs $H_1,\ldots,H_k$, there is a graph $G$ such that $G
\xrightarrow{\textrm{ind}}(H_1,\ldots,H_k)$, and the {\it induced
Ramsey number} $r_{\textrm{ind}}(H_1,\ldots,H_k)$ is the minimum
number of vertices in such $G$. If all $H_i=H$, then we denote
$r_{\textrm{ind}}(H_1,\ldots,H_k)=r_{\textrm{ind}}(H;k)$.

Early proofs of the Induced Ramsey Theorem give weak bounds on these
numbers. For two colors, the more recent results \cite{FoSu},
\cite{KoPrRo} significantly improve these estimates. However, it
seems that the approaches in those papers do not generalize to give
good results for many colors. There is a simple way of giving an
upper bound on the multicolor induced Ramsey
$r_{\textrm{ind}}(H_1,\ldots,H_k)$ in terms of induced Ramsey
numbers with fewer colors. Notice that if $G_1
\xrightarrow{\textrm{ind}}(H_1,\ldots,H_\ell)$, $G_2
\xrightarrow{\textrm{ind}}(H_{\ell+1},\ldots,H_k)$, and $G
\xrightarrow{\textrm{ind}} (G_1,G_2)$, then $G
\xrightarrow{\textrm{ind}} (H_1,\ldots,H_k)$. (To see this, just
group together the first $\ell$ colors and the last $k-\ell$
colors.) For fixed $H$, this gives that $r_{\textrm{ind}}(H;k)$
grows at most like a tower of $2$s of height roughly $\log k$. The
following result improves considerably on this tower bound.

\begin{theorem}\label{inducedRamsey}
For every graph $H$ there is a constant $c(H)$ such that $r_{\textrm{ind}}(H;k) \leq k^{c(H)k}$ for every integer $k \geq 2$.
\end{theorem}

For $H$ on $n$ vertices, the proof shows that $c(H)$ can be taken to
be $500n^3$. It is worth mentioning that as a function of $k$ (up to
the constant $c(H)$), the upper bound in Theorem \ref{inducedRamsey}
is similar to the best known estimate for ordinary Ramsey numbers.
On the other hand, it is known and easy to show that in general
these numbers grow at least exponentially in $k$. The proof of the
above theorem combines ideas used to establish bounds on Ramsey
numbers of graphs with bounded chromatic number together with some
properties of pseudo-random graphs.

\vspace{0.3cm} \noindent {\bf Organization of the paper.}\,\, In the
next section we present our key ideas and techniques and illustrate
them on a simple example, the proof of Theorem \ref{main1}. More
involved applications of these techniques which require additional
ideas are given in Sections \ref{section3}-\ref{section4}. There we
prove results on bipartite degenerate graphs, graphs with bounded
chromatic number, and subdivided graphs, respectively. In Section
\ref{embeddinginducedgraphs}, we prove a useful embedding lemma for
induced subgraphs which we apply in Section 7 together with our
basic techniques to obtain two results on the Erd\H{o}s-Hajnal
conjecture. In Section 8, we apply this lemma again to show that
$H$-free graphs satisfying a rather weak edge density condition
contain a very dense linear-sized induced subgraph. We then use this
fact about $H$-free graphs in Section 9 to prove two results on
disjoint edge patterns in simple topological graphs. In Section
\ref{multicolorinduced}, we prove Theorem \ref{inducedRamsey} which
gives an upper bound on multicolor induced Ramsey numbers. The last
section of this paper contains some concluding remarks together with
a few conjectures and open problems. Throughout the paper, we
systematically omit floor and ceiling signs whenever they are not
crucial for the sake of clarity of presentation. We also do not make
any serious attempt to optimize absolute constants in our statements
and proofs.

\section{Dependent random choice and graph embeddings}\label{section1}

The purpose of this section is to illustrate on the simplest
example, the proof of Theorem \ref{main1}, the key ideas and
techniques that we will use. The first tool is a simple yet
surprisingly powerful lemma whose proof uses a probabilistic
argument known as {\it dependent random choice}. Early versions of
this technique were developed in the papers \cite{Go2,KoRo,Su}.
Later, variants were discovered and applied to various Ramsey and
density-type problems (see, e.g., \cite{KoSu,AlKrSu,Su1,KoRo2}).

This lemma demonstrates that every dense graph contains a large set
of vertices $A$ with the useful property that {\it almost all} small
subsets of $A$ have many common neighbors. The earlier applications
of dependent random choice for Ramsey-type problems (e.g.,
\cite{KoRo, Su, KoSu, AlKrSu, KoRo2}) required that {\it all} small
subsets of $A$ have large common neighborhood. This stronger
assumption, which is possible to obtain using dependent random
choice, allows one to use a simple greedy procedure to embed sparse
graphs. However, the price of achieving this stronger property is
rather high, since the resulting set $A$ has a sublinear number of
vertices in the order of the graph. Consequently, one cannot use
this to prove a linear upper bound on Ramsey numbers. Our main
contribution here shows how to circumvent this difficulty. The
second tool, Lemma \ref{embeddinglemma}, is an embedding result for
hypergraphs. It can be used to embed sparse bipartite graphs without
requiring all subsets of $A$ to have large common neighborhood.

For  a vertex $v$ in a graph $G$, let $N(v)$ denote the set of
neighbors of $v$ in $G$. Given a subset $U \subset G$, the {\it
common neighborhood $N(U)$ of $U$} is the set of all vertices of $G$
that are adjacent to $U$, i.e., to {\it every} vertex in $U$.
Sometimes, we write $N_G(U)$ to stress that the underlying graph is
$G$ when this is not entirely clear from the context. By a $d$-set,
we mean a set of cardinality $d$. The following lemma demonstrates
that every dense bipartite graph contains a large set of vertices
$A$ such that almost every $d$-set in $A$ has many common neighbors.

\begin{lemma}\label{lem:dependent}
If $\epsilon>0$ and $G=(V_1,V_2;E)$
is a bipartite graph with $|V_1|=|V_2|=N$ and at least $\epsilon N^2$ edges,
then for all positive integers $a,d,t,x$, there is a subset $A \subset V_2$ with $|A| \geq 2^{-1/a}\epsilon^tN$
such that for all but at most $2\epsilon^{-ta}\left(\frac{x}{N}\right)^t\left(\frac{|A|}{N}\right)^a{N \choose d}$
$d$-sets $S$ in  $A$, we have $|N(S)| \geq x$.
\end{lemma}
\begin{proof}
Let $T$ be a subset of $t$ random vertices of $V_1$, chosen
uniformly with repetitions. Set $A=N(T)$, and let $X$ denote the
cardinality of $A \subset V_2$. By linearity of expectation and by convexity of $f(z)=z^t$,
$$\mathbb{E}[X]=\sum_{v \in V_2}\left(\frac{|N(v)|}{N}\right)^t=
N^{-t}\sum_{v \in V_2}|N(v)|^t \geq N^{1-t}\left(\frac{\sum_{v \in
V_1}|N(v)|}{N}\right)^t \geq \epsilon^tN.$$

Let $Y$ denote the random variable counting the number of $d$-sets
in $A$ with fewer than $x$ common neighbors. For a given $d$-set
$S$, the probability that $S$ is a subset of $A$ is
$\left(\frac{|N(S)|}{N}\right)^{t}$. Therefore, we have
$$\mathbb{E}[Y] \leq {N \choose d}\left(\frac{x-1}{N}\right)^t .$$
By convexity, $\mathbb{E}[X^a] \geq
\mathbb{E}[X]^a$. Thus, using linearity of expectation, we obtain
$$\mathbb{E}\left[X^a-\frac{\mathbb{E}[X]^a}{2\mathbb{E}[Y]}\, Y-\frac{\mathbb{E}[X]^a}{2}\right] \geq 0.$$
Therefore, there is a choice of $T$ for which this expression is nonnegative. Then
$$X^a \geq \frac{1}{2}\mathbb{E}[X]^a \geq \frac{1}{2}\epsilon^{ta}N^a$$
and $$Y\leq 2X^a\mathbb{E}[Y]\mathbb{E}[X]^{-a}<
2\epsilon^{-ta}\left(\frac{x}{N}\right)^t\left(\frac{|A|}{N}\right)^a{N
\choose d}.$$
This implies $|A|=X\geq 2^{-1/a}\epsilon^tN$, completing the proof.
\end{proof}

A {\it hypergraph} ${\cal F}=(V,E)$ consists of a vertex set $V$ and an
edge set $E$, which is a collection of subsets of $V$. It is
{\it down-closed} if $e_1 \subset e_2$ and $e_2 \in E$ implies $e_1
\in E$. The following lemma shows how to embed a sparse
hypergraph in a very dense hypergraph.

\begin{lemma}\label{embeddinglemma}
Let ${\cal H}$ be a $n$-vertex hypergraph with maximum degree $d$
such that each edge of $\cal H$ has size at most $h$. If ${\cal
F}=(V,E)$ is a down-closed hypergraph with $N \geq 4n$ vertices and
more than $(1-(4d)^{-h}){N \choose h}$ edges of cardinality $h$,
then there are at least $(N/2)^n$ labeled copies of $\cal H$ in
${\cal F}$.
\end{lemma}
\begin{proof}
Call a subset $S \subset V$ of size $|S| \leq h$ {\it good} if $S$ is
contained in more than $\big(1-(4d)^{|S|-h}\big){N \choose h-|S|}$ edges of
$\cal F$ of cardinality $h$. For a good set $S$ with $|S|<h$ and a vertex
$j \in V \setminus S$, call $j$ {\it bad} with respect to $S$ if $S
\cup \{j\}$ is not good. Let $B_S$ denote the
set of vertices $j \in V \setminus S$ that are bad with respect to $S$.
The key observation is that if $S$ is good with $|S|<h$, then
$|B_S| \leq N/(4d)$. Indeed, suppose $|B_S|>N/(4d)$, then the
number of $h$-sets containing $S$ that are not edges of $G$ is at
least
$$\frac{|B_S|}{h-|S|}(4d)^{|S|+1-h}{N \choose h-|S|-1} >(4d)^{|S|-h}{N
\choose h-|S|},$$ which contradicts the fact that $S$ is good.

Fix a labeling $\{v_1,\ldots,v_n\}$ of the vertices of $\cal H$.
Since the maximum degree of $\cal H$ is $d$, for every vertex $v_i$
there are at most $d$ subsets $S \subset L_i=\{v_1,\ldots,v_i\}$
containing $v_i$ such that $S=e \cap L_i$ for some edge $e$ of $\cal
H$. We use induction on $i$ to find many embeddings $f$ of $\cal H$
in $\cal F$ such that for each edge $e$ of $H$, the set $f(e \cap
L_i)$ is good.

By our definition, the empty set is good. Assume at step $i$, for
all edges  the sets $f(e \cap L_i)$ are good. There are at most $d$
subsets $S$ of $L_{i+1}$ that are of the form $S=e \cap L_{i+1}$
where $e$ is an edge of $\cal H$ containing $v_{i+1}$. By the
induction hypothesis, for each such subset $S$, the set $f(S
\setminus \{v_{i+1}\})$ is good and therefore there are at most
$\frac{N}{4d}$ bad vertices in $\cal F$ with respect to it. In total
this gives at most $d\frac{N}{4d}=N/4$ vertices. The remaining at
least $3N/4-i$ vertices in ${\cal F}\setminus f(L_i)$ are good with
respect to all the above sets $f(S \setminus \{v_{i+1}\})$ and we
can pick any of them to be $f(v_{i+1})$. Notice that this
construction guarantees that $f(e \cap L_{i+1})$ is good for every
edge $e$ in  $\cal H$. In the end of the process we obtain a mapping
$f$ such that $f(e \cap L_n)=f(e)$ is good for every $e$ in $\cal
H$. In particular, $f(e)$ is contained in at least one edge of $\cal
F$ of cardinality $h$ and therefore $f(e)$ itself is an edge of
$\cal F$ since $\cal F$ is down-closed. This shows that $f$ is
indeed an embedding of $\cal H$ in $\cal F$. Since at step $i$ we
have at least $3N/4-i$ choices for vertex $v_{i+1}$ and since $N
\geq 4n$, we get at least $\prod_{i=0}^{n-1}(\frac{3}{4}N-i) \geq
(N/2)^n$ labeled copies of $\cal H$.
\end{proof}

Using these two lemmas we can now complete the proof of Theorem
\ref{main1}, which implies also Theorem \ref{main} and Corollaries
\ref{cormain}, \ref{cubecor}. For a graph $G$ and a subset $A$, we
let $G[A]$ denote the subgraph of $G$ induced by $A$. If $G=(V,E)$
is a graph with $N$ vertices and $\epsilon{N \choose 2}$ edges,
then, by averaging over all partitions $V=V_1 \cup V_2$ with
$|V_1|=|V_2|=N/2$, we can find a partition with at least
$\epsilon(N/2)^2$ edges between $V_1$ and $V_2$. Hence, Theorem
\ref{main1} follows from the following statement.

\begin{theorem}\label{mainturan}
Let $H$ be a bipartite graph with parts $U_1$ and $U_2$, $n$ vertices and
maximum degree at most $d \geq 2$. If $\epsilon>0$ and
$G=(V_1,V_2;E)$ is a bipartite graph with $|V_1|=|V_2|=N \geq
16d\epsilon^{-d}n$ and at least $\epsilon N^2$ edges, then $G$
contains at least $(32d)^{-n/2}\epsilon^{dn}N^n$ labeled copies of
$H$.
\end{theorem}
\begin{proof}
Assume without loss of generality that $|U_2| \geq |U_1|$. Let
$\cal H$ be the hypergraph with vertex set $U_2$ such that a subset $D
\subset U_2$ is an edge of $\cal H$ if and only if there is a vertex
$u \in U_1$ with $N_H(u)=D$. This $\cal H$ has $|U_2| \leq
n$ vertices, maximum degree at most $d$ and edges of size at most $d$.

Let $x=\frac{\epsilon^d}{8d}N$, so in particular, $x \geq 2n \geq
4|U_1|$. We show that $G$ contains many copies of $H$ so that the
vertices of $U_i$ are embedded in $V_i$ for $i \in \{1,2\}$. Call a
$d$-set $S \subset V_2$ {\it nice} if $|N_G(S)| \geq x$. Let $\cal
F$ be the down-closed hypergraph with vertex set $V_2$ whose edges
are all subsets of $V_2$ which are contained in a nice $d$-set. An
important observation is that each copy of $\cal H$ in $\cal F$ can
be used to embed many distinct copies of $H$ in $G$ as follows.
Suppose that $f: U_2 \rightarrow V_2$ is an embedding of $\cal H$ in
$\cal F$. For every copy of $H$ use $f$ to embed vertices in $U_2$.
Embed vertices in $U_1$ one by one. Suppose that the current vertex
to embed is $u \in U_1$ and let $D$ be the set of neighbors of $u$
in $U_2$. Then $D$ is an edge of $\cal H$ and therefore $f(D)$ is
contained in a nice set and has at least $x$ common neighbors in
$G$. Since only at most $|U_1|$ of them can be occupied by other
vertices of the copy of $H$ which we are embedding, we still have at
least $x-|U_1|\geq \frac{3}{4} x$ available vertices to embed $u$.
Since this holds for every vertex in $U_1$, altogether we get at
least $\left(\frac{3}{4} x\right)^{|U_1|}$ distinct embeddings of
$H$ for each copy of $\cal H$ in $\cal F$.

Next we will find a large induced
subhypergraph of $\cal F$ which is sufficiently dense to apply Lemma \ref{embeddinglemma}.
By Lemma \ref{lem:dependent} with $a=t=d$, $V_2$ contains  a subset
$A$ of size $|A| \geq 2^{-1/d}\epsilon^d N \geq 2^{-1/2}\epsilon^d N$ such that the number
of $d$-sets $S \subset A$ satisfying $|N_G(S)| < x$ is at most
\begin{eqnarray*}\label{bound} 2\epsilon^{-d^2}\left(\frac{x}{N}\right)^d\left(\frac{|A|}{N}\right)^d{N
\choose d} &=&
2\epsilon^{-d^2}\left(\frac{\epsilon^{d}N/8d}{N}\right)^d\left(\frac{|A|}{N}\right)^d{N
\choose d}\\
&=& 2(8d)^{-d}\left(\frac{|A|}{N}\right)^d{N \choose d}\\
&<& (4d)^{-d}{|A| \choose d}.
\end{eqnarray*}
Here we use that $|A|^d \leq 2^{d-1}d!{|A| \choose d}$ which follows from $d \geq 2$
and $|A| \geq 2^{-1/d}\epsilon^d N > 8d$.

Applying Lemma \ref{embeddinglemma} with $h=d$, to the subhypergraph
${\cal F}[A]$ induced by the set $A$, we obtain at least
$\left(\frac{|A|}{2}\right)^{|U_2|}$ labeled copies of $\cal H$. By
the above discussion each such copy of $\cal H$ can be extended to
$\left(\frac{3}{4} x\right)^{|U_1|}$ labeled copies of $H$.
Therefore, using that $|U_1| \leq |U_2|$, $|U_1|+|U_2|=n$, $|A| \geq
2^{-1/2}\epsilon^d N$ and $x=\frac{\epsilon^d}{8d}N$, we conclude
that $G$ contains at least
$$
\left(\frac{|A|}{2}\right)^{|U_2|}\left(\frac{3}{4}x\right)^{|U_1|}
\geq \left(\frac{3}{32}\right)^{-|U_1|}
\big(2^{-3/2}\big)^{-|U_2|}\, d^{-n/2}\epsilon^{dn}N^n \geq
(32d)^{-n/2}\epsilon^{dn}N^n$$ labeled copies of $H$, completing the
proof.
\end{proof}

\section{Degenerate bipartite graphs}\label{section3}
The main result of this section is the following theorem which
implies Theorem \ref{degenerate}.

\begin{theorem}\label{mainturan1}
Let $H$ be a $d$-degenerate bipartite graph with $n$ vertices and
maximum degree $\Delta$. Let $G=(V_1,V_2;E)$ be a bipartite graph
with $|V_1|=|V_2|=N$ vertices and at least $\epsilon N^2$ edges.
Suppose $d \geq 2, d/n \leq \delta \leq 1$ and let
$x=2^{-9}\epsilon^{\left(1+(1+\delta^{-1})d\right)(1+\delta)}\Delta^{-\delta}N$.
If $x \geq  4n$, then $G$ contains at least $(x/4)^n$ labeled copies
of $H$.
\end{theorem}

\noindent To obtain from this statement Theorem \ref{degenerate},
recall that every graph with $N$ vertices and $\epsilon{N \choose
2}$ edges, has a partition $V=V_1 \cup V_2$ with $|V_1|=|V_2|=N/2$
such that the number of edges between $V_1$ and $V_2$ is at least
$\epsilon(N/2)^2$. Moreover, our result shows that if $H$ is a
bipartite $d$-degenerate graph of order $n$ and maximum degree at
most exponential in $d$, then every large graph $G$ with edge
density $\epsilon$ contains at least a fraction $\epsilon^{O(dn)}$
of all possible copies of $H$. This is best possible up to the
constant factor in the exponent and shows that Sidorenko's
conjecture discussed in Section \ref{multiplicity} is not very far
from being true.

\vspace{0.15cm} \noindent {\bf Proof of Theorem \ref{mainturan1}.}\,
Let $t=(1+\delta^{-1})d$ and $u=t+d$. By Lemma \ref{lem:dependent}
with parameters $a=1, u, t, x$, $V_1$ contains a subset $A'$ with
$|A'| \geq \frac{1}{2}\epsilon^t N$ such that the number $Y$ of
$u$-sets $T \subset A'$ with $|N(T)| < x$ is at most
$$Y \leq 2\epsilon^{-t}\left(\frac{x}{N}\right)^t\left(\frac{|A'|}{N}\right) {N \choose u} \leq 2\epsilon^{-t}
\left(\frac{x}{N}\right)^t{N \choose u}.$$

Let $S$ be a random subset of $A'$ of size $t$ and let $A_2=N(S)$.
Denote by $Q$ the random variable counting the number of $u$-sets $T
\subset A'$ containing $S$ such that $|N(T)| <x$. Note that the
number of $u$-sets $T$ with $|N(T)| <x$  is at most $Y$ and each of
them contains the random subset $S$ with probability ${u \choose
t}/{|A'| \choose t}$. Thus, using that $t-d=\delta^{-1}d$, $u=t+d,
\frac{u^u}{u!}<e^u<2^{3u/2}$ and $|A'| \geq 2 x \geq u$, we have
\begin{eqnarray*}\mathbb{E}[Q] &=& \frac{{u \choose t}}{{|A'|
\choose t}}Y \leq \frac{{u \choose t}}{{|A'| \choose t}}
2\epsilon^{-t}\left(\frac{x}{N}\right)^t{N \choose u} \leq
2\left(\frac{ux}{\epsilon|A'|N}\right)^t{N \choose u}\\ &\leq&
2\left(\frac{ux}{\frac{1}{2}\epsilon^{t+1}N^2}\right)^t{N \choose u}
 \leq
2^{t+1}\frac{u^t}{u!}\epsilon^{-(t+1)t}\left(\frac{x}{N}\right)^{t-d}x^d\\
&=& 2^{t+1}\frac{u^u}{u!}2^{-9d/\delta}\Delta^{-d}\frac{x^d}{u^d}
<\frac{1}{2}(2\Delta)^{-d}{x \choose d} .\end{eqnarray*} It is
important to observe that $Q$ also gives an upper bound on the
number of $d$-sets $T'$ in $A'\setminus S$ which have less than $x$
common neighbors in $A_2$. Indeed, we can correspond to every such
$T'$ a set $T=T' \cup S$. Since $N(T)= N(S) \cap N(T')=A_2 \cap
N(T')$, $T$ has less than $x$ common neighbors. Therefore the number
of sets $T'$ is bounded by the number of sets $T$. Let
$A_1=A'\setminus S$. Then, using that $t=d+\delta^{-1}d \leq 2n \leq
x/2$, we have that $|A_1|=|A'|-|S| \geq 2x-t\geq x$.

Let $Z$ denote the random variable counting the number of subsets of
$A_2$ with cardinality $d$ with less than $x$ common neighbors in
$A_1$. Note that such a set has at most $t+x \leq 2x$ common
neighbors in $A'$. For a given $d$-set $R \subset V_2$, the
probability that $R$ is a subset of $A_2$ is ${|N_{A'}(R)| \choose
t}{|A'| \choose t}^{-1} \leq \left(
\frac{|N_{A'}(R)|}{|A'|}\right)^t$. Therefore, using that
$t=d+\delta^{-1}d$, we have
\begin{eqnarray*}\mathbb{E}[Z] &<& {N \choose
d}\left(\frac{2x}{|A'|}\right)^t<\frac{N^d}{d!}\left(\frac{2x}{\frac{1}{2}\epsilon^tN}\right)^t
=2^{2t}\epsilon^{-t^2}\left(\frac{x}{N}\right)^{t-d} \frac{x^d}{d!}\\
&=& 2^{2t-9d/\delta}\epsilon^t\Delta^{-d}\frac{x^d}{d!} <
\frac{1}{2}(2\Delta)^{-d}{x \choose d} .\end{eqnarray*} Since $Q$
and $Z$ are nonnegative discrete random variables, by Markov's
inequality, $\mathbb{P}\big[Q>2\mathbb{E}[Q]\big]<1/2$ and
$\mathbb{P}\big[Z>2\mathbb{E}[Z]\big]<1/2$. Thus there is a choice
of set $S$ such that $$Q \leq 2\mathbb{E}[Q] < (2\Delta)^{-d}{x
\choose d}$$ and
$$Z \leq 2\mathbb{E}[Z]<(2\Delta)^{-d}{x \choose d}.$$

Since $Q<{x \choose d}$ and $|A_1| \geq x$, then there is a $d$-set
in $A_1$ that has at least $x$ common neighbors in $A_2$ and so
$|A_2|\geq x$. Therefore, for each $i \in \{1,2\}$, $|A_i| \geq x$
and all but less than $(2\Delta)^{-d}{x \choose d}$ subsets of $A_i$
of size $d$ have at least $x$ common neighbors in $A_{3-i}$. By
Lemma \ref{degeneracyhelpful} applied to the induced subgraph of $G$
by $A_1 \cup A_2$, we have that $G$ contains at least $(x/4)^n$
labeled copies of $H$. \hfill $\Box$

\begin{lemma}\label{degeneracyhelpful}
Let $H=(U_1,U_2;F)$ be a $d$-degenerate bipartite graph with $n$
vertices and maximum degree $\Delta$. Let $G=(A_1,A_2;E)$ be a
bipartite graph such that for $i \in \{1,2\}$, $|A_i| \geq x \geq
4n$ and the number of $d$-sets $U \subset A_i$ with $N(U) < x$ is
less than $(2\Delta)^{-d}{x \choose d}$. Then $G$ contains at least
$(x/4)^{n}$ labeled copies of $H$.
\end{lemma}
\begin{proof}
A $d$-set $S \subset A_i$ is {\it good} if $|N(S)| \geq x$,
otherwise it is {\it bad}. Also, a subset $U \subset A_i$ with $|U|
< d$ is {\it good} if it is contained in less than
$(2\Delta)^{|U|-d}{x \choose d-|U|}$ bad subsets of $A_i$ of size
$d$. A vertex $v \in A_i$ is {\it bad with respect to a subset $U
\subset A_i$} with $|U|<d$ if $U$ is good but $U \cup \{v\}$ is not.
Note that, for any good subset $U \subset A_i$ with $|U|<d$, there
are at most $\frac{x}{2\Delta}$ vertices that are bad with respect
to $U$. Indeed, if not, then there would be more than
$$\frac{x/(2\Delta)}{d-|U|}(2\Delta)^{|U|+1-d}{x \choose d-|U|-1}\geq (2\Delta)^{|U|-d}{x\choose d-|U|}$$
subsets of $A_i$ of size $d$ containing $U$ that are bad, which
would contradict $U$ being good.

Since $H$ is $d$-degenerate, then there is an ordering
$\{v_1,\ldots,v_n\}$ of the vertices of $H$ such that each vertex
$v_i$ has at most $d$ neighbors $v_j$ with $j<i$. Let $N^-(v_i)$ be
all the neighbors $v_j$ of $v_i$ with $j<i$. Let
$L_h=\{v_1,\ldots,v_h\}$. We will use induction on $h$ to find at
least $(x/4)^{n}$ embeddings $f$ of $H$ in $G$ such that $f(U_i)
\subset A_i$ for $i \in \{1,2\}$ and for every vertex $v_j$ and
every $h \in [n]$, the set $f(N^-(v_j)\cap L_h)$ is good.

By our definition, the empty set is good for each $i \in \{1,2\}$.
We will embed the vertices in the increasing order of their indices.
Suppose we are embedding $v_h$. Then, by the induction hypothesis,
for each vertex $v_j$, the set $f(N^-(v_j)\cap L_{h-1})$ is good.
Since the set $f(N^-(v_h\cap L_{h-1}))=f(N^-(v_h))$ is good, it has
at least $x$ common neighbors. Also, $v_h$ has degree at most
$\Delta$, so there are at most $\Delta$ sets $f(N_-(v_j) \cap
L_{h-1})$ where $v_j$ is a neighbor of $v_h$ and $j>h$. These sets
are good, so there are at most $\Delta \frac{x}{2\Delta} = x/2$
vertices which are bad for at least one of them. This implies that
there at least $x-x/2-(h-1)>x/4$ vertices in the common neighborhood
of $f(N^-(v_h))$ which are not occupied yet and are good for all the
above sets $f(N_-(v_j) \cap L_{h-1})$. Any of these vertices can be
chosen as $f(v_h)$. Altogether, we get at least $(x/4)^n$ labeled
copies of $H$.
\end{proof}

This proof can be modified to obtain the bound $r(H) \leq 2^{cp}n$
for $p$-arrangeable bipartite $H$, where $c$ is some absolute
constant. Note that the maximum degree $\Delta$ of $H$ is only used
in the last paragraph to bound the number of sets $f(N_-(v_j) \cap
L_{h-1})$ where $v_j$ is a neighbor of $v_h$ and $j>h$. As we
already discussed in detail in the introduction if graph $H$ is
$p$-arrangeable then there is an ordering of its vertices for which
the number of distinct sets $N_-(v_j) \cap L_{h-1}$ where $v_j$ is a
neighbor of $v_h$ and $j>h$ is bounded by $2^{p-1}$ for every $h$.
Therefore, we can use for $p$-arrangeable bipartite graphs the same
proof as for $p$-degenerate bipartite graphs with maximum degree at
most $2^{p-1}$. We easily obtain the following slight variant of
Lemma \ref{degeneracyhelpful} for the proof.

\begin{lemma}\label{arrangeablehelpful}
Let $H=(U_1,U_2;F)$ be a $p$-arrangeable bipartite graph with $n$
vertices. Let $G=(A_1,A_2;E)$ be a bipartite graph such that for $i
\in \{1,2\}$, $|A_i| \geq x \geq 4n$ and the number of $p$-sets $U
\subset A_i$ with $N_G(U) < x$ is less than $2^{-p^2}{x \choose p}$.
Then $G$ contains at least $(x/4)^{n}$ labeled copies of $H$.
\end{lemma}
The remaining details of the proof are essentially identical and
therefore omitted.

\section{Graphs with bounded chromatic number}\label{section2}

The following result implies Theorem \ref{chromatic} since every
graph with chromatic number $q$ and maximum degree $d$ satisfies $q
\leq d+1$ and hence $(2d+2)(2q-3)+2 \leq 4dq$. Moreover, Theorem
\ref{t41} shows that every $2$-edge-coloring of $K_N$ with $N \geq
2^{4dq}n$ contains at least $2^{-4dqn}N^{n}$ labeled monochromatic
copies of any $n$-vertex graph $H$ with chromatic number $q$ and
maximum degree $d$. This implies that the Ramsey multiplicity for
graphs with fixed chromatic number and whose average degree is at
least a constant fraction of the maximum degree is not very far from
the bound given by a random coloring.

\begin{theorem}
\label{t41} If $H$ is a graph with $n$ vertices, chromatic number
$q$, and maximum degree $d \geq 2$, then for every $2$-edge-coloring
of $K_N$ with $N \geq 2^{(2d+2)(2q-3)+2}n$, there are at least
$\left(2^{-(2d+2)(2q-3)-2}N\right)^n$ labeled monochromatic copies
of $H$.
\end{theorem}
\begin{proof}
Consider a $2$-edge-coloring of of $K_N$ with colors $0$ and $1$.
For $j \in \{0,1\}$, let $G_j$ denote the graph of color $j$. Let
$A_1$ be the  vertex set of $K_N$ and $x=2^{-(2d+2)(2q-3)}N$, so $x
\geq 4n$. We will pick subsets $A_1 \supset A_2 \supset \ldots
\supset A_{2q-2}$ such that for each $i \leq 2q-3$, we have
$|A_{i+1}| \geq |A_i|/2^{2d+2}$ and there is a color $c(i) \in
\{0,1\}$ such that there are less than $(2d)^{-d}{x \choose d}$
$d$-sets $U \subset A_{i+1}$ which have less than $x$ common
neighbors in the induced subgraph $G_{c(i)}[A_i]$.

Given $A_i$, we can pick $c(i)$ and $A_{i+1}$ as follows.
Arbitrarily partition $A_i$ into two subsets $A_{i,1}$ and $A_{i,2}$
of equal size. Let $c(i)$ denote the densest of the two colors
between $A_{i,1}$ and $A_{i,2}$. By Lemma \ref{lem:dependent} with
$\epsilon=1/2$, $a=1$, and $t=2d$, there is a subset $A_{i+1}
\subset A_{i,2} \subset A_i$ with $|A_{i+1}| \geq
2^{-2d-1}|A_{i,2}|=2^{-2d-2}|A_i|$ such that for all but at most
$$2 \cdot
2^{2d}\left(\frac{x}{|A_i|}\right)^{2d}\left(\frac{|A_{i+1}|}{|A_{i,2}|}\right){|A_{i,2}|
\choose d}\leq 2^{2d+1}\left(\frac{x}{|A_i|}\right)^{2d}{|A_i|/2
\choose d}<2^{d+1}\left(\frac{x}{|A_i|}\right)^d\frac{x^d}{d!}
<(2d)^{-d}{x \choose d}$$ $d$-sets $U \subset A_{i+1}$, $U$ has at
least $x$ common neighbors in $G_{c(i)}[A_i]$. Here, the last
inequality uses the fact that $|A_i| \geq 2^{-(i-1)(2d+2)}N \geq
2^{-(2q-4)(2d+2)}N=2^{2d+2}x$.

Given the subsets $A_1 \supset \ldots \supset A_{2q-2}$  with the
desired properties and the colors $c(1),\ldots,c(2q-3)$, notice that
$|A_{2q-2}| \geq 2^{-(2d+2)(2q-3)}N = x$. By the pigeonhole
principle, one of the two colors is represented at least $q-1$ times
in the sequence $c(1),\ldots,c(2q-3)$. Without loss of generality
suppose that $0$ is this popular color. Let $V_1=A_1$, and for $1
\leq k < q$, let $V_{k+1}=A_{j+1}$, where $j$ is the
$k^{\textrm{th}}$ smallest positive integer such that $c(j)=0$. By
applying Lemma \ref{chromatichelpful} below to the graph $G_0$ and
subsets $V_1,\ldots,V_q$, we can find at least $(x/4)^n$ labeled
monochromatic copies of $H$, which completes the proof.
\end{proof}

\begin{lemma}\label{chromatichelpful}
Suppose $G$ is a graph with vertex set $V_1$, and let $V_1 \supset
\ldots \supset V_q$ be a family of nested subsets of $V_1$ such that
$|V_q| \geq x \geq 4n$, and for $1 \leq i <q$, all but less than
$(2d)^{-d}{x \choose d}$ $d$-sets $U \subset V_{i+1}$ satisfy $|N(U)
\cap V_{i}| \geq x$. Then, for every $q$-partite graph $H$ with $n$
vertices and maximum degree at most $d$, there are at least
$(x/4)^n$ labeled copies of $H$ in $G$.
\end{lemma}
\begin{proof}
A $d$-set $S \subset V_{i+1}$ is {\it good with respect to $i$} if
$|N(S) \cap V_{i}| \geq x$, otherwise it is {\it bad with respect to
$i$}. Also, a subset $U \subset V_{i+1}$ with $|U| < d$ is {\it good
with respect to $i$} if there are less than $(2d)^{|U|-d}{x \choose
d-|U|}$ subsets of $V_{i+1}$ of size $d$ that contain $U$ and are
bad with respect to $i$. For a good subset $U \subset V_{i+1}$ with
respect to $i$ with $|U|<d$, call a vertex $v \in V_{i+1}$ {\it bad
with respect to $U$ and $i$} if $U \cup \{v\}$ is bad with respect
to $i$. For any $i$ and subset $U\subset V_{i+1}$ that is good with
respect to $i$, there are less than $\frac{x}{2d}$ bad vertices with
respect to $U$ and $i$. Indeed, if otherwise, then the number of
subsets of $V_{i+1}$ of size $d$ containing $U$ that are bad is at
least
$$\frac{x/(2d)}{d-|U|}(2d)^{|U|+1-d}{x \choose d-|U|-1} \geq (2d)^{|U|-d}{x \choose d-|U|},$$
which contradicts the fact that $U$ is good with respect to $i$.

Consider a partition  $W_1 \cup \ldots \cup W_q$ of the vertices of
$H$ into $q$ independent sets. Order the vertices
$\{v_1,\ldots,v_n\}$ of $H$ such that the vertices of $W_i$ precede
the vertices of $W_j$ for $i>j$. Let $L_h=\{v_1,\ldots,v_h\}$. For a
vertex $v_j$, let $N^-(v_j)$ denote the set of vertices $v_i, i<j$
adjacent to $v_j$ and $N^+(v_j)$ denote the set of vertices $v_i,
i>j$ adjacent to $v_j$. By our ordering of the vertices of $H$ and
the fact that each $W_k$ is an independent set, if $w \in W_k$, $v
\in N^{-}(w)$, and $v \in W_{\ell}$, then $\ell>k$. Similarly, if $w
\in W_k$, $v \in N^{+}(w)$, and $v \in W_{\ell}$, then $\ell<k$.

We use induction on $h$ to find many embeddings $f$ of $H$ in $G$
such that $f(W_k) \subset V_k$ for all $k$, and the set $f(L_h \cap
N^-(w))$ is {\it good} with respect to $k$ for all $h$, $k$, and $w
\in W_k$. Since $f(W_i) \subset V_i$ and the sets $V_i$ are nested,
by the above discussion we also have that $f(N^-(w))\subset V_{k+1}$
for all $w \in W_k$. By our definition, the empty set is good with
respect to every $k$, which demonstrates the base case $h=0$ of the
induction. We pick the vertices for the embedding in order of their
index. Suppose we are embedding $v_h$ with $v_h \in W_\ell$. Our
induction hypothesis is that we have already embedded $L_{h-1}$ with
the desired properties, so for each $k$ and $w\in W_k$, the set
$f(L_{h-1} \cap N^-(w)) \subset V_{k+1}$ is good with respect to
$k$. We need to show how to pick $f(v_h) \in V_\ell$ that is not
already occupied such that $f(v_h)$ is adjacent to all vertices in
$f(N^-(v_h))$ and for each vertex $w \in N^+(v_h)$ with $w \in W_j$,
$f(v_h)$ is good with respect to $f(N^-(w) \cap L_{h-1})$ and $j$.

Since $f(N^-(v_h)\cap L_{h-1})=f(N^-(v_h))$ is good with respect to
$\ell$, then $f(N^-(v_h))$ is contained in a $d$-set that is good
with respect to $\ell$ and so it has at least $x$ common neighbors
in $V_{\ell}$. Let $w \in N^+(v_h)$ such that $w \in W_j$, then
$j<\ell$. Since $V_{\ell} \subset V_{j+1}$, then there are less than
$\frac{x}{2d}$ vertices in $V_\ell$ that are bad with respect to
$f(N^-(w) \cap L_{h-1})$ and $j$. Since there are at most $d$ such
$w$, then there are at least $x-d\frac{x}{2d}-(h-1) \geq x/4$
unoccupied vertices in $V_\ell$ satisfying the above properties,
which we can choose for $f(v_h)$. Altogether, we get at least
$(x/4)^n$ copies of $H$ in $G$.
\end{proof}

The constant factor in the exponent in Theorems \ref{t41} and
\ref{chromatic} can be improved for large $q$ by roughly a factor of
$2$ by picking $t \approx d+\log d$ instead of $t=2d$. Also, the
above proof can be easily adapted  to give the following upper bound
on multicolor Ramsey numbers.

\begin{theorem}
If $H_1,\ldots,H_k$ are $k \geq 2$ graphs with at most $n$ vertices,
chromatic number at most $q$, and maximum degree at most $\Delta$,
then $$r(H_1,\ldots,H_k) \leq k^{2k\Delta q}n.$$
\end{theorem}

\section{Density theorem for subdivided graphs}\label{section4}
Note that the $1$-subdivision of a graph $\Gamma$ is a bipartite
graph whose first part contains the vertices of $\Gamma$ and whose
second part contains the vertices which were used to subdivide the
edges of $\Gamma$. Furthermore, the vertices in the second part have
degree two. Also, if $\Gamma$ has $n$ edges and no isolated vertices
then its $1$-subdivision has at most $3n$ vertices. Therefore,
Theorem \ref{subdivided1} follows from the following theorem.

\begin{theorem}\label{subdivided}
If $H=(U_1,U_2;F)$ is a bipartite graph with $n$ vertices such that
every vertex in $U_2$ has degree $2$, $G$ is a graph with $2N$
vertices, $2\epsilon N^2$ edges, and $N \geq 128\epsilon^{-3}n$,
then $H$ is a subgraph of $G$.
\end{theorem}
\begin{proof}
By averaging over all partitions $V=V_1 \cup V_2$ of $G$ with
$|V_1|=|V_2|=N$, we can find a partition with at least $\epsilon
N^2$ edges between $V_1$ and $V_2$. Delete the vertices of $V_1$
with less than $\epsilon N/2$ neighbors in $V_2$, and let $V_1'$
denote the set of remaining vertices of $V_1$. Note that we deleted
at most $\epsilon N^2/2$ edges so between $V'_1$ and $V_2$ there are
still at least $\epsilon N^2/2$ edges. Let $G'$ be the graph with
parts $V'_1, V_2$ and all edges between them. Every vertex in $V_1'$
has degree at least $\epsilon N/2$ in $G'$ and $|V_1'| \geq \epsilon
N/2$.

Let $H'$ be the graph with vertex set $U_1$ such that two vertices
in $U_1$ are adjacent in $H'$ if and only if they have a neighbor in
common. Since $|U_2|+|U_1|=n$, then the number of edges of $H'$ is
at most $n$. Consider an auxiliary graph $G^*$ with vertex set
$V_1'$ such that two vertices of $V_1'$ are adjacent if their common
neighborhood  in $G'$ has cardinality at least $n$. Note that given
an embedding $f: U_1 \rightarrow V_1$ of $H'$ in $G^*$, we can
extend it to an embedding of $H$ in $G'$ as follows. Use $f$ to
embed vertices in $U_1$. Embed vertices in $U_2$ one by one. Suppose
that the current vertex to embed is $u \in U_2$ and let $D$ be the
set of neighbors of $u$ in $U_1$, so $|D|=2$. Then $D$ is an edge in
$H'$ and so $f(D)$ is an edge of $G^*$. Therefore, $f(D)$ has at
least $n$ common neighbors in $G'$. As the total number of vertices
of $H$ embedded so far is less than $n$, one of the common neighbors
of $f(D)$ is still unoccupied and can be used to embed $u$. Thus it
is enough to find a copy of $H'$ in $G^*$.

To do this, we construct a family of nested subsets $V_1'=A_0
\supset A_1 \supset \ldots \supset \ldots$ such that for all $i \geq
1$, $|A_i| \geq \frac{\epsilon}{8}|A_{i-1}|$ and the maximum degree
in the complement of the induced subgraph $G^*[A_i]$ is at most
$(\epsilon/8)^i|A_i|$. Set $c_i=(\epsilon/8)^i$ and let $E_i$ be the
set of edges of $\bar G^*[A_i]$. Then $|E_i| \leq c_i|A_i|^2/2$.

Having already picked $A_1,\ldots,A_{i-1}$ satisfying the above two
desired properties, we show how to pick $A_i$. Let $w$ be a vertex
from $V_2$ chosen uniformly at random. Let $A$ denote the
intersection of $A_{i-1}$ with the neighborhood of $w$, and $X$ be
the random variable denoting the cardinality of $A$. Since every
vertex in $V_1'$ has degree at least $\epsilon N/2$,
$$\mathbb{E}[X]=\sum_{v \in A_{i-1}}\frac{|N_{G'}(v)|}{|V_2|}\geq \frac{\epsilon}{2}|A_{i-1}|.$$

Let $Y$ be the random variable counting the number of pairs in $A$
with fewer than $n$ common neighbors in $V_2$, i.e., $Y$ counts the
number of pairs in $A$ that are not edges of $G^*$. Notice that the
probability that a pair $R$ of vertices of $A_{i-1}$ is in $A$ is at
most $\frac{|N_{G'}(R)|}{|V_2|}$. Recall that $E_{i-1}$ is the set
of all pairs $R$ in $A_{i-1}$ with $|N_{G'}(R)|< n$ (these are edges
of $\bar G^*$) and $|E_{i-1}| \leq c_{i-1}|A_{i-1}|^2/2$.
Therefore, we have
$$\mathbb{E}[Y] < \frac{n}{N}|E_{i-1}| \leq \frac{n}{N}\frac{c_{i-1}}{2}|A_{i-1}|^2.$$
By convexity, $\mathbb{E}[X^2] \geq \mathbb{E}[X]^2$. Thus, using
linearity of expectation, we obtain
 $$\mathbb{E}\left[X^2-\frac{\mathbb{E}[X]^2}{2\mathbb{E}[Y]}\, Y-\mathbb{E}[X]^2/2\right] \geq 0.$$
Therefore, there is a choice of $w$ such that this expression is
nonnegative. Then
$$X^2 \geq \frac{1}{2}\mathbb{E}[X]^2 \geq \frac{\epsilon^2}{8}|A_{i-1}|^2$$ and
$$Y \leq 2\frac{X^2}{\mathbb{E}[X]^2}\mathbb{E}[Y] \leq 4\epsilon^{-2} c_{i-1} \frac{n}{N}X^2
\leq \frac{\epsilon}{16} c_{i-1}\frac{X^2}{2}.$$ From the first
inequality, we have $|A|=X \geq \frac{\epsilon}{4}|A_{i-1}|$ and the
second inequality implies that the average degree in the induced
subgraph $\bar G^*[A]$ is at most $\epsilon c_{i-1}|A|/16$. If $A$
contains a vertex of degree more than $\epsilon c_{i-1}|A|/16$, then
delete it, and continue this process until the remaining induced
subgraph of $\bar G^{*}[A]$ has maximum degree at most $\epsilon
c_{i-1}|A|/16$. Let $A_i$ denote the vertex set of this remaining
induced subgraph. Clearly, the number of deleted edges is at least
$(|A|-|A_i|)\epsilon c_{i-1}|A|/16$. As explained above, the number
of edges of $G^*[A]$ is at most $\epsilon c_{i-1}|A|^2/32$, so we
arrive at the inequality $|A|-|A_i| \leq |A|/2$. Hence, $|A_i| \geq
|A|/2 \geq \frac{\epsilon}{8}|A_{i-1}|$ and the maximum degree in
$\bar G^* [A_i]$ is at most $\frac{\epsilon}{16}c_{i-1}|A| \leq
\frac{\epsilon}{8}c_{i-1}|A_i| = c_i|A_i|$. Therefore, we have shown
how to find the nested family of subsets with the desired
properties.

Label the vertices $\{v_1,\ldots,v_{|U_1|}\}$ of $H'$ in decreasing
order of their degree. Since $H'$ has at most $n$ edges, the degree
of $v_i$ is at most $2n/i$. We will find an embedding $f$ of $H'$ in
$G^*$ which embeds vertices in the order of their index $i$. The
vertex $v_i$ will be embedded in $A_j$ where $j$ is the least
positive integer such that $c_j \leq \frac{i}{4n}$. Since $c_j =
(\frac{\epsilon}{8})^j$, then
$$|A_j| \geq c_j \, |A_0| \geq  c_j \, \frac{\epsilon N}{2} \geq
\frac{\epsilon}{8} \, \frac{i}{4n} \, \frac{\epsilon N}{2} \geq
2i.$$ Assume we have already embedded the vertices $\{v_k;  k <i\}$
and we want to embed $v_i$. Let $N^-(v_i)$ be the set of vertices
$v_k, k<i$ that are adjacent to $v_i$ in $H'$. The maximum degree in
the induced subgraph $\bar G^*[A_j]$ is at most $c_j|A_j| \leq
\frac{i}{4n}|A_j|$. Since $v_i$ has degree at most $\frac{2n}{i}$ in
$H'$, then at least $|A_j|-\frac{2n}{i}\cdot \frac{i}{4n}|A_j| \geq
|A_j|/2$ vertices of $A_j$ are adjacent in $G^*$ to all the vertices
in $f(N^-(v_i))$. Since also $|A_j|/2 \geq i$, then there is a
vertex in $A_j \setminus f(\{v_1,\ldots,v_{i-1}\})$ that is adjacent
in $G^*$ to all the vertices of $f(N^-(v_i))$. Use this vertex to
embed $v_i$ and continue. This gives a copy of $H'$ in $G^*$,
completing the proof.
\end{proof}

\section{Embedding induced subgraphs}\label{embeddinginducedgraphs}

To prove the results stated in Sections \ref{sub:ErdosHajnal} -
\ref{sub:inducedRamsey}, we need the following embedding lemma for
induced subgraphs.

\begin{lemma}\label{inducedembedding}
Let $G$ and $F$ be two edge-disjoint graphs on the same vertex set
$U$ and let $A_1 \supset \ldots \supset A_n$ be vertex subsets of
$U$ with $|A_n| \geq m \geq 2n$ for some positive integers $m$ and
$n$. Suppose that for every $i<n$, all but less than $(2n)^{-2n}{m
\choose n}^2$ pairs $(S_1,S_2)$ of disjoint subsets of $A_{i+1}$
with $|S_1|=|S_2|=n$ have at least $m$ vertices in $A_i$ that are
adjacent to $S_1$ in $G$ and are adjacent to $S_2$ in $F$. Then, for
each graph $H$ with $n$ vertices $V=\{v_1,\ldots,v_n\}$, there is an
embedding $f:V \rightarrow U$ such that for every pair $i<j$,
$(f(v_i),f(v_j))$ is an edge of $G$ if $(v_i,v_j)$ is an edge of
$H$, and $(f(v_i),f(v_j))$ is an edge of $F$ if $(v_i,v_j)$ is not
an edge of $H$.
\end{lemma}
\begin{proof}
Call a pair $(S_1,S_2)$ of disjoint subsets of $A_{i+1}$ with
$|S_1|=|S_2|=n$ {\it good with respect to $i$} if there are at least
$m$ vertices in $A_i$ that are adjacent to $S_1$ in $G$ and adjacent
to $S_2$ in $F$, otherwise it is {\it bad with respect to $i$}.
Also, call a pair $(U_1,U_2)$ of disjoint subsets of $A_{i+1}$ each
of cardinality at most $n$ {\it good with respect to $i$} if less
than $(2n)^{|U_1|+|U_2|-2n}{m \choose n-|U_1|}{m \choose n-|U_2|}$
pairs $(S_1,S_2)$ of disjoint subsets of $A_{i+1}$ with
$|S_1|=|S_2|=n$, $U_1 \subset S_1$, and $U_2 \subset S_2$ are bad
with respect to $i$, otherwise it is {\it bad with respect to $i$}.
Note that if $(U_1,U_2)$ is good with respect to $i$, then there is
a pair $(S_1,S_2)$ of disjoint subsets of $A_{i+1}$ with
$|S_1|=|S_2|=n$, $U_1 \subset S_1$ and $U_2 \subset S_2$ that is
good with respect to $i$, so
$$|N_G(U_1) \cap N_F(U_2) \cap A_i| \geq |N_G(S_1) \cap N_F(S_2)
\cap A_i| \geq m.$$

For $b \in \{1,2\}$ and a pair $(U_1,U_2)$ of subsets of $A_{i+1}$
that is good with respect to $i$ with $|U_b|<n$, call a vertex $w
\in A_{i+1}$ {\it bad with respect to $(U_1,U_2,b,i)$} if $b=1$ and
$(U_1 \cup \{w\},U_2)$ is bad with respect to $i$, or if $b=2$ and
$(U_1,U_2 \cup \{w\})$ is bad with respect to $i$. For $b \in
\{1,2\}$ and a pair $(U_1,U_2)$ of subsets of $A_{i+1}$ that is good
with respect to $i$, there are less than $\frac{m}{2n}$ vertices $w
\in A_{i+1}$ that are bad with respect to $(U_1,U_2,b,i)$. Indeed,
otherwise the number of pairs $(S_1,S_2)$ of subsets of $A_{i+1}$
each of size $n$ with $U_1 \subset S_1$ and $U_2 \subset S_2$ that
are bad with respect to $i$ is at least
$$\frac{m/(2n)}{n-|U_b|}(2n)^{|U_{3-b}|+|U_b|+1-2n}{m \choose
n-|U_b|-1}{m \choose n-|U_{3-b}|} \geq (2n)^{|U_1|+|U_2|-2n}{m
\choose n-|U_1|}{m \choose n-|U_2|},$$ which contradicts the fact
that $(U_1,U_2)$ is good with respect to $i$.

We next show how to find a copy of $H$ in $G$ such that vertex pairs
in this copy corresponding to nonedges of $H$ are edges of $F$. We
embed the vertices of $H$ one by one in the increasing order of
their index. Let $L_h=\{v_1,\ldots,v_{h}\}$. For a vertex $v_j$, let
$N^-(v_j)$ denote the vertices $v_i$ adjacent to $v_j$ with $i<j$
and $N^+(v_j)$ denote the vertices $v_i$ adjacent to $v_j$ with
$i>j$. We use induction on $h$ to construct the embedding $f$ of $H$
 such that $f(v_j) \in A_{n-j+1}$ for all $j$, and for every
$v_j$ and $h<j$, the pair $\left(f(L_h \cap N^-(v_j)),f(L_h
\setminus N^-(v_j))\right)$ is good with respect to $n-j+1$.

The induction hypothesis is that we have already embedded $L_{h-1}$
and for every $v_j$, the pair $\left(f(L_{h-1} \cap
N^-(v_j)),f(L_{h-1} \setminus N^-(v_j))\right)$ is good with respect
to $n-j+1$. In the base case $h=1$, the induction hypothesis holds
since our definition implies that the pair $(\emptyset,\emptyset)$
is good with respect to $j$ for every $j$. Since the sets are
nested, we have $f(L_{h-1}) \subset A_{k}$ for any $k \leq n-h+2$.
We need to show how to pick $f(v_h) \in V_{n-h+1}$ that is not
already occupied and satisfies
\begin{itemize}
\item $f(v_h)$ is adjacent to $f(N^-(v_h))$ in $G$ and adjacent to $f(L_{h-1} \setminus N^-(v_h))$ in $F$,
\item for each vertex $v_j \in N^+(v_h)$, $f(v_h)$ is
not bad with respect to $$\left(f(N^-(v_j) \cap L_{h-1}),f(L_{h-1}
\setminus N^-(v_j)),1,n-j+1\right),~\textrm{and}$$
\item for each vertex $v_j \not \in N^+(v_h)$ with $j>h$, $f(v_h)$ is
not bad with respect to $$\left(f(N^-(v_j) \cap L_{h-1}),f(L_{h-1}
\setminus N^-(v_j)),2,n-j+1\right).$$
\end{itemize}

Since $\left(L_{h-1} \cap f(N^-(v_h)),f(L_{h-1} \setminus
N^-(v_h))\right)$ is good with respect to $n-h+1$, then there are at
least $m$ vertices in $A_{n-h+1}$ that are adjacent to every vertex
of $f(N^-(v_h))=f(L_{h-1} \cap N^-(v_h))$ in $G$ and are adjacent to
every vertex of $f(L_{h-1} \setminus N^-(v_h))$ in $F$. For each
$v_j \in N^+(v_h)$, there are less than $\frac{m}{2n}$ vertices of
$A_{n-h+1}$ that are bad with respect to $\left(f(N^-(v_j) \cap
L_{h-1}),f(L_{h-1} \setminus N^-(v_j)),1,n-j+1\right).$ Also, for
each $v_j \not \in N^+(v_h)$ with $j>h$, there are less than
$\frac{m}{2n}$ vertices of $A_{n-h+1}$ that are bad with respect to
$\left(f(N^-(v_j) \cap L_{h-1}),f( L_{h-1} \setminus
N^-(v_j)),2,n-j+1\right)$. Since the number of $v_j$ with $j>h$ is
$n-h$ and the number of already occupied vertices is $h-1$, then
there are at least
$$m-(n-h)\frac{m}{2n}-(h-1) > m/2-(n-1) \geq 1$$ unoccupied vertices to choose
for $f(v_h) \in A_{n-h+1}$ satisfying the above three desired
properties, which, by induction on $h$, completes the proof.
\end{proof}

A graph is {\it $n$-universal} if it contains all graphs on $n$
vertices as induced subgraphs. For the proofs of Theorems
\ref{halferdoshajnal} - \ref{bipartitedisjoint} and Corollary
\ref{cor:drawing}, we need the special case $F=\bar G$ of the above
lemma, which is stated below.

\begin{corollary}\label{cor:inducedembed}
Let $m$ and $n$ be positive integers and let $A_1 \supset \ldots
\supset A_n$ be vertex subsets of a graph $G$ with $|A_n| \geq m
\geq 2n$. If for all $i < n$, all but less than $(2n)^{-2n}{m
\choose n}^2$ pairs $(S_1,S_2)$ of disjoint subsets of $A_{i+1}$
with $|S_1|=|S_2|=n$ have at least $m$ vertices in $A_i$ that are
adjacent to all vertices in $S_1$ and no vertices in $S_2$, then
graph $G$ is $n$-universal.
\end{corollary}

\section{Ramsey-type results for $H$-free graphs}\label{erdoshajnal1}

The purpose of this section is to prove Theorems
\ref{halferdoshajnal} and \ref{offdiagonal} which are related to the
Erd\H{o}s-Hajnal conjecture. We first give an overview of the proofs
before jumping into the details.

Lemma \ref{lemmaforhalf} below demonstrates that for a (large
enough) graph $G$ that is not too sparse and does not contain a pair
of large subsets with edge density almost $1$ between them, there is
a large vertex subset $A$ with the property that almost all pairs
$(S_1,S_2)$ of disjoint subsets of $A$ of size $n$ satisfy that
$|N_G(S_1) \cap N_{\bar G}(S_2)|$ is large. The first step in the
proof of Lemma \ref{lemmaforhalf} uses Lemma \ref{lem:dependent} to
get a large subset $A$ for which almost all vertex subsets $S_1$ of
size $n$ have large common neighborhood. Using the fact that $G$
does not contain a pair of large subsets with edge density almost
$1$ between them, we show that for almost all pairs $(S_1,S_2)$ of
subsets of $A$ of size $n$, $|N_G(S_1) \cap N_{\bar G}(S_2)|$ is
large.

By repeated application of Lemma \ref{lemmaforhalf} and an
application of Corollary \ref{cor:inducedembed}, we arrive at Lemma
\ref{offdiagonalhelpful1}, which says that every graph is
$n$-universal, or contains a large independent set, or has two large
subsets with edge density almost $1$ between them. The deductions of
Theorems \ref{halferdoshajnal} and \ref{offdiagonal} from Lemma
\ref{offdiagonalhelpful1} are relatively straightforward.

\begin{lemma}\label{lemmaforhalf}
Suppose $z$ is a positive integer, $\beta,\epsilon>0$, and $G=(V,E)$
is a graph on $N$ vertices and at least $\beta {N \choose 2}$ edges
such that for each pair $(W_1,W_2)$ of disjoint subsets of $V$ each
of cardinality at least $z$, there is a vertex in $W_1$ with less
than $(1-2\epsilon)|W_2|$ neighbors in $W_2$. If $2 \leq n \leq z$
and $m$ satisfy
$$4nz^{1/2n}N^{1-1/2n} \leq m \leq \frac{\beta^2 \epsilon^{2n} N}{16n},$$
then there is a subset $A \subset V$ with $|A| \geq
\frac{1}{4}\beta^{4n}N$ such that all but less than $(2n)^{-2n}{m
\choose n}^2$ pairs $(S_1,S_2)$ of disjoint subsets of $A$ with
$|S_1|=|S_2|=n$ have at least $m$ vertices of $G$ adjacent to every
vertex in $S_1$ and no vertex in $S_2$.
\end{lemma}
\begin{proof}
By averaging over all partitions $V=V_1 \cup V_2$ of $G$ with
$|V_1|=|V_2|=N/2$, we can find a partition with at least $\beta
(N/2)^2$ edges between $V_1$ and $V_2$. By Lemma \ref{lem:dependent}
with $a=1$, $t=4n$, $d=n$, and $x=\epsilon^{-n}m$, there is a subset
$A \subset V_2$ with cardinality at least
$\frac{1}{2}\beta^{4n}|V_2|=\frac{1}{4}\beta^{4n}N$ such that for
all but at most
\begin{equation}\label{eq:7.1}
2\beta^{-4n}\left(\frac{x}{N/2}\right)^{4n}\left(\frac{|A|}{N/2}\right){N/2
\choose n} \leq \left(\frac{2x}{\beta N}\right)^{4n}\frac{N^n}{n!}
\end{equation} subsets $S_1$ of $A$ of size $n$, we have $|N_G(S_1)| \geq x$.

If $G$ contains (not necessarily disjoint) subsets $B_1,B_2$ each of
cardinality at least $2z$ such that every vertex in $B_1$ is
adjacent to at least $(1-\epsilon)|B_2|$ vertices in $B_2$, then
letting $W_1$ be any $z$ vertices of $B_1$ and $W_2=B_2 \setminus
W_1$, we have a contradiction with the hypothesis of the lemma.
Indeed, $|W_2|\geq |B_2|/2 \geq z$ and every vertex of $W_1$ is
adjacent to at least $|W_2|-\epsilon|B_2| \geq (1-2\epsilon)|W_2|$
vertices in $W_2$.

Let $S_1$ be a subset of $A$ of cardinality $n$ with $|N_G(S_1)|
\geq x$. We will show that almost all subsets $S_2$ of $A$ of
cardinality $n$ satisfy $|N_{G}(S_1) \cap N_{\bar G}(S_2)| \geq m$.
The number of vertices $u_1$ of $A$ such that
$$|N_{G}(S_1) \cap N_{\bar G}(u_1)| < \epsilon|N_{G}(S_1)|$$ is at most $2z$,
otherwise each of these at least $2z$ vertices has at least
$(1-\epsilon)|N_{G}(S_1)|$ neighbors in $N_G(S_1)$, which by the
above discussion would contradict the hypothesis of the lemma. Pick
any vertex $u_1 \in A$ such that $$|N_{G}(S_1) \cap N_{\bar G}(u_1)|
\geq \epsilon|N_{G}(S_1)|.$$ After picking $u_1,\ldots,u_i$ such
that
$$|N_{G}(S_1) \cap N_{\bar G}(\{u_1,\ldots,u_i\})| \geq
\epsilon^{i}|N_{G}(S_1)|,$$ again there are at most $2z$ vertices
$u_{i+1}$ such that
$$|N_{G}(S_1) \cap N_{\bar G}(\{u_1,\ldots,u_i,u_{i+1}\})| <
\epsilon^{i+1}|N_{G}(S_1)|,$$ otherwise each of these at least $2z$
vertices has at least $(1-\epsilon)|N_{G}(S_1) \cap N_{\bar
G}(\{u_1,\ldots,u_i\})|$ neighbors in $N_{G}(S_1) \cap N_{\bar
G}(\{u_1,\ldots,u_i\})$, which by the above discussion would
contradict the hypothesis of the lemma. Note that during this
process for every index $i$ there are at least
$|A|-|S_1|-(i-1)-2z>|A|-2z-2n$ choices for $u_i \in A \setminus S_1$
not already chosen. Therefore, given $S_1$ with $|N_G(S_1)| \geq x$,
we conclude that the number of ordered $n$-tuples $(u_1,\ldots,u_n)$
of distinct vertices of $A\setminus S_1$ with
$$|N_G(S_1) \cap N_{\bar G}(\{u_1,\ldots,u_n\})| \geq \epsilon^{n}x=m$$ is at least
$$(|A|-2z-2n)^n \geq (|A|-4z)^n \geq |A|^n-4nz|A|^{n-1}.$$
Hence, the number of (unordered) subsets $S_2=\{u_1,\ldots,u_n\}$ of
$A \setminus S_1$ with $|N_{G}(S_1) \cap N_{\bar G}(S_2)| < m$ is at
most $4zn|A|^{n-1}/n!$. This implies that the number of disjoint
pairs $S_1,S_2 \subset A$ with $|S_1|=|S_2|=n$, $|N_G(S_1)|\geq x$,
and $|N_{G}(S_1) \cap N_{\bar G}(S_2)| < m$ is at most ${|A| \choose
n} \cdot \frac{1}{n!}4zn|A|^{n-1}$. Also, notice that by
(\ref{eq:7.1}) the number of disjoint pairs $S_1,S_2 \subset A$ with
$|S_1|=|S_2|=n$ and $|N_G(S_1)|<x$ is at most $\left(\frac{2x}{\beta
N}\right)^{4n}\frac{N^n}{n!} \cdot {|A| \choose n}$.

Therefore, the number of pairs of disjoint subsets $S_1,S_2 \subset
A$ with $|N_{G}(S_1) \cap N_{\bar G}(S_2)| < m$ and $|S_1|=|S_2|=n$
is at most
\begin{equation}\label{ineqind1} \left(\frac{2x}{\beta
N}\right)^{4n}\frac{N^n}{n!}{|A| \choose n}+{|A| \choose
n}\frac{1}{n!}4zn|A|^{n-1}.\end{equation} Using the upper bound on
$m$ and $|A| \leq N/2$, we have
\begin{eqnarray}
\label{ineqind2} \left(\frac{2x}{\beta
N}\right)^{4n}\frac{N^n}{n!}{|A| \choose n} &\leq&
2^{3n}\left(\frac{\epsilon^{-n}m}{\beta
N}\right)^{4n}\frac{N^{2n}}{n!^2} =
2^{3n}\beta^{-4n}\epsilon^{-4n^2}\left(\frac{m}{N}\right)^{2n}\frac{m^{2n}}{n!^2}\nonumber\\
&\leq&
2^{3n}\beta^{-4n}\epsilon^{-4n^2}\left(\frac{\beta^2\epsilon^{2n}}{16n}\right)^{2n}\frac{m^{2n}}{n!^2}=2^{-5n}n^{-2n}\frac{m^{2n}}{n!^2}
\nonumber \\ &<& \frac{1}{2}(2n)^{-2n}{m \choose n}^2.
\end{eqnarray}
Using the lower bound on $m$ and $|A| \leq N/2$, we have
\begin{eqnarray}\label{ineqind3}
{|A| \choose n}\frac{1}{n!}4zn|A|^{n-1} &\leq& n!^{-2}4zn|A|^{2n-1}
\leq 2^{3-2n}n!^{-2}nzN^{2n-1} \leq
2^{3-2n}n!^{-2}n\left(\frac{m}{4n}\right)^{2n} \nonumber\\ &\leq&
n2^{3-4n}(2n)^{-2n}\frac{m^{2n}}{n!^2} \leq \frac{1}{2}(2n)^{-2n}{m
\choose n}^2.
\end{eqnarray}

Combining (\ref{ineqind1}), (\ref{ineqind2}), and (\ref{ineqind3}),
we have that there are less than $(2n)^{-2n}{m \choose n}^2$ pairs
of disjoint subsets $S_1,S_2 \subset A$ with $|S_1|=|S_2|=n$ and
$|N_{G}(S_1) \cap N_{\bar G}(S_2)| < m$, completing the proof.
\end{proof}

\vspace{0.1cm} \noindent The next lemma follows from repeated
application of Lemma \ref{lemmaforhalf} and an application of
Corollary \ref{cor:inducedembed}.

\begin{lemma}\label{offdiagonalhelpful1}
Let $\epsilon>0$, $H$ be a graph on $n \geq 3$ vertices and
$G=(V,E)$ be an $H$-free graph with $N$ vertices and no independent
set of size $t$ with $N \geq (4t)^{8n^3}\epsilon^{-4n^2}n$. Then
there is a pair $W_1,W_2$ of disjoint subsets of $V$ such that
$|W_1|,|W_2| \geq (4t)^{-8n^3}\epsilon^{4n^2}N$ and every vertex in
$W_1$ is adjacent to all but at most $2\epsilon|W_{2}|$ vertices of
$W_{2}$.
\end{lemma}
\begin{proof}
Since $G$ has no independent set of size $t$, then by Tur\'an's
theorem (see, e.g., \cite{AlSp},\cite{Di}) every induced subgraph of
$G$ with $v \geq t^2$ vertices has at least $\frac{v^2}{2t}$ edges.
Let $z=(4t)^{-8n^3}\epsilon^{4n^2}N$, so $z \geq n$.

Suppose for contradiction that there are no disjoint subsets
$W_1,W_2$ with $|W_1|,|W_2| \geq z$ and every vertex in $W_1$
adjacent to all but at most $2\epsilon|W_2|$ vertices of $W_2$. Fix
$\beta=\frac{1}{2t}$ and
$$m=\frac{\beta^2\epsilon^{2n}(\frac{1}{4}\beta^{4n})^{n-1}N}{16n},$$
and repeatedly apply Lemma \ref{lemmaforhalf} $n-1$ times (note that
the choice of parameters allows this). We get a family of nested
subsets $V=A_1 \supset \ldots \supset A_{n}$ such that $|A_n| \geq
\left(\frac{1}{4}\beta^{4n}\right)^{n-1}N \geq m \geq 2n$ and for $1
\leq i \leq n-1$, all but less than $(2n)^{-2n}{m \choose n}^2$
pairs $(S_1,S_2)$ of disjoint subsets of $A_{i+1}$ with
$|S_1|=|S_2|=n$ have at least $m$ vertices in $A_i$ in the common
neighborhood of $S_1$ in $G$ and the common neighborhood of $S_2$ in
$\bar G$. By Corollary \ref{cor:inducedembed}, $G$ contains $H$ as
an induced subgraph, contradicting the assumption that $G$ is
$H$-free, and completing the proof.
\end{proof}

From Lemma \ref{offdiagonalhelpful1}, we quickly deduce Theorem
\ref{halferdoshajnal}, which says that for every $H$ there is
$c=c(H)>0$ such that any $H$-free graph of order $N$ contains a
complete bipartite graph with parts of size $N^{c}$ or an
independent set of size $N^{c}$.

\vspace{0.2cm}\noindent {\bf Proof of Theorem
\ref{halferdoshajnal}:} Let $H$ be a graph on $n$ vertices, $G$ be a
$H$-free graph on $N$ vertices, and
$t=\frac{1}{10}N^{\frac{1}{10n^3}}$. If $G$ has no independent set
of size $t$, then by Lemma \ref{offdiagonalhelpful1} with
$\epsilon=\frac{1}{4t}$, $G$ must contain disjoint subsets $W_1$ and
$W_2$ each of cardinality at least $2t$ such that every vertex of
$W_1$ is adjacent to all but at most $\frac{1}{2t}|W_2|$ vertices in
$W_2$. Picking $t$ vertices in $W_1$ and their common neighborhood
in $W_2$, which has size at least $|W_2|-t\frac{1}{2t}|W_2| \geq
|W_2|/2 \geq t$, shows that $G$ contains $K_{t,t}$ and completes the
proof. \qed

\vspace{0.1cm} We are now ready to prove Theorem \ref{offdiagonal},
which says that for every $H$-free graph $G$ of order $N$ and
$n_1,n_2$ satisfying $\log n_1 \log n_2 \leq c(H)\log N$, $G$
contains a clique of size $n_1$ or an independent set of size $n_2$.
For a graph $G$, the clique number is the order of the largest
complete subgraph of $G$ and the independence number is the order of
the largest independent set of $G$. Let $\omega_{t,n}(N)$ be the
minimum clique number over all graphs with $N$ vertices and
independence number less than $t$ that are not $n$-universal.

\vspace{0.15cm}\noindent {\bf Proof of Theorem \ref{offdiagonal}:}
Let $t=n_2$, $H$ be a graph on $n$ vertices, and $G$ be an $H$-free
graph with $N$ vertices, no independent set of size $n_2$, and
clique number $\omega_{t,n}(N)<n_1$. Since $\bar G$ is $\bar
H$-free, we may suppose without loss of generality that $n_2 \geq
n_1$. By Lemma \ref{offdiagonalhelpful1} with $t=n_2$ and
$\epsilon=\frac{1}{4t}$, there are disjoint subsets $W_1$ and $W_2$
of $V$, each of size at least $(4t)^{-8n^3-4n^2}N \geq
(4t)^{-10n^3}N$, such that every vertex in $W_1$ is adjacent to all
but at most $\frac{1}{2t}|W_2|$ vertices in $W_2$. Pick a largest
clique $X$ in $W_1$. The cardinality of clique $X$ is less than $n_1
\leq n_2=t$ by assumption. So $|X| < t$ and at least half of the
vertices of $W_2$ are adjacent to $X$. Pick a largest clique $Y$ in
the vertices of $W_2$ adjacent to $X$. The clique number of $G$ is
at least $|X|+|Y|$. Hence,
$$\omega_{t,n}(N) \geq \omega_{t,n}(|W_1|)+\omega_{t,n}(|W_2|/2) \geq 2\omega_{t,n}((8t)^{-10n^3}N).$$
Let $d$ be the largest integer such that $N \geq (8t)^{10n^3d}$, so
$d+1 \geq \frac{1}{10n^3}\frac{\log N}{\log 8t}$. We have
$\omega_{t,n}(N) \geq 2^d$ by repeated application of the inequality
above. Hence,
$$\log n_1 \log n_2 \geq \log \omega_{t,n}(N) \log t \geq d\log t
\geq \frac{1}{20n^3}\frac{\log N}{\log 8t} \log t \geq
\frac{1}{80n^3} \log N,$$ completing the proof. \qed

\section{Edge distribution of $H$-free graphs}

As we already mentioned in the introduction, there are several
results which show that the edge distribution of $H$-free graphs is
far from being uniform. One such result, obtained by R\"odl, says
that for every graph $H$ and $\epsilon \in (0,1/2)$, there is a
positive constant $\delta=\delta(\epsilon,H)$ such that any $H$-free
graph on $N$ vertices contains an induced subgraph on at least
$\delta N$ vertices with edge density either at most $\epsilon$ or
at least $1-\epsilon$. In \cite{FoSu}, we gave an alternative proof
which gives a much better bound on $\delta(\epsilon,H)$. Combining
our techniques with the approach of \cite{FoSu}, we obtain a
generalization of R\"odl's theorem which shows that a seemingly weak
edge density condition is sufficient for an $H$-free graph to
contain a very dense linear-sized induced subgraph. For $\delta \in
(0,1]$ and a monotone increasing function $\beta:(0,1] \rightarrow
(0,1]$, we call a graph on $N$ vertices {\it $(\beta,\delta)$-dense}
if every induced subgraph on $\sigma N$ vertices has edge density at
least $\beta(\sigma)$ for $\sigma \geq \delta$.

\begin{theorem}\label{largedensity}
For each monotone increasing function $\beta:(0,1] \rightarrow
(0,1]$, $\epsilon >0$, and graph $H$, there is
$\delta=\delta(\beta,\epsilon,H)>0$ such that every
$(\beta,\delta)$-dense $H$-free graph on $n$ vertices contains an
induced subgraph on at least $\delta n$ vertices with edge density
at least $1-\epsilon$.
\end{theorem}

Notice that R\"odl's theorem is the special case of this statement
when $\beta$ is the constant function with value $\epsilon$. An
important step in the proof of Theorem \ref{largedensity} is the
following lemma which shows how to find two large vertex subsets
with edge density almost $1$ between them in a
$(\beta,\delta)$-dense $H$-free graph.

\begin{lemma}\label{betalemma}
Let $\beta:(0,1] \rightarrow (0,1]$ be a monotone increasing
function, $\epsilon>0$, and $H$ be a graph on $n$ vertices. There is
$\delta>0$ such that every $H$-free graph $G=(V,E)$ on $N$ vertices
that is $(\beta,\delta)$-dense contains disjoint subsets $V_1,V_2
\subset V$ each of cardinality at least $\delta N$ such that every
vertex in $V_1$ is adjacent to all but at most $2\epsilon|V_2|$
vertices in $V_2$.
\end{lemma}
\begin{proof}
Define the sequence $\{\delta_i\}_{i=1}^n$ of real numbers in
$(0,1]$ recursively as follows: $\delta_1=1$ and
$\delta_i=\frac{1}{4}\beta^{4n}(\delta_{i-1}) \, \delta_{i-1}$. Let
$\delta=(\frac{\epsilon^{2n}\delta_{n}}{64n^{2}})^{2n}n^{-1}$,
$z=n\delta N$, and $m=\frac{\epsilon^{2n}\delta_n}{8n}N$, so
$$4nz^{\frac{1}{2n}}N^{1-\frac{1}{2n}} = m/2 \leq m =
\frac{\epsilon^{2n}\delta_n}{8n}N =
\frac{\epsilon^{2n}\beta^{4n}(\delta_{n-1}) \, \delta_{n-1}}{32n}N
\leq \frac{\beta^2(\delta_{n-1})\epsilon^{2n} \cdot
\delta_{n-1}N}{16n}.$$ Since $G$ is $(\beta,\delta)$-dense, we have
$N \geq \delta^{-1}$ so that $z \geq n$ and $m \geq 2n$.

Suppose for contradiction that $G$ does not contain a pair $V_1,V_2$
of disjoint vertex subsets each of cardinality at least $z$ such
that every vertex in $V_1$ is adjacent to all but at most
$2\epsilon|V_2|$ vertices in $V_2$. By repeated application of Lemma
\ref{lemmaforhalf} $n-1$ times (note that the choice of parameters
allows this), we find a family of nested subsets $V=A_1 \supset
\ldots \supset A_{n}$ with all $|A_i| \geq \delta_iN$ and $|A_n|
\geq \delta_nN \geq m \geq 2n$ which have the following property.
For all $i<n$, all but less than $(2n)^{-2n}{m \choose n}^2$ pairs
$(S_1,S_2)$ of subsets of $A_{i+1}$ with $|S_1|=|S_2|=n$ have at
least $m$ vertices in $A_i$ adjacent to all vertices in $S_1$ and no
vertices in $S_2$. By Corollary \ref{cor:inducedembed}, $G$ contains
$H$ as an induced subgraph, contradicting the assumption that $G$ is
$H$-free, and completing the proof.
\end{proof}

The final step of the proof of Theorem \ref{largedensity} is to show
how to go from two vertex subsets with edge density almost $1$
between them as in Lemma \ref{betalemma} to one vertex subset with
edge density almost $1$. To accomplish this, we use the key lemma in
\cite{FoSu}. We first need some definitions. For a graph $G=(V,E)$
and disjoint subsets $W_1,\ldots,W_t \subset V$, the {\it density}
$d_{G}(W_1,\ldots,W_t)$ between the $t \geq 2$ vertex subsets
$W_1,\ldots,W_t$ is defined by
$$d_G(W_1,\ldots,W_t)=\frac{\sum_{ i < j}
e(W_i,W_j)}{\sum_{i < j } |W_i||W_j|},$$ where $e(A,B)$ is the
number of pairs $(a,b) \in A \times B$ that are edges of $G$.

\begin{definition} For $\alpha,\rho,\epsilon \in [0,1]$ and positive integer $t$, a
graph $G=(V,E)$ is {\bf $(\alpha,\rho,\epsilon,t)$-dense} if, for
all subsets $U \subset V$ with $|U| \geq \alpha |V|$, there are
disjoint subsets $W_{1},\ldots,W_{t} \subset U$ with $|W_{1}|=\ldots
=|W_{t}|=\lceil \rho |U|\rceil$ and $d_{G}(W_1,\ldots,W_{t})\geq
1-\epsilon$.
\end{definition}

By averaging, if $\alpha'\geq \alpha$, $\rho' \leq \rho$,
$\epsilon'\geq \epsilon$, $t' \leq t$, and $G$ is
$(\alpha,\rho,\epsilon,t)$-dense, then $G$ is also
$(\alpha',\rho',\epsilon',t')$-dense. The key lemma in \cite{FoSu}
(applied to the complement of the graph) says that if a graph is
$(\frac{1}{2}\alpha\rho,\rho',\epsilon,t)$-dense and
$(\alpha,\rho,\epsilon/4,2)$-dense, then it is also
$(\alpha,\frac{1}{2}\rho\rho',\epsilon,2t)$-dense.

\vspace{0.2cm}\noindent {\bf Proof of Theorem \ref{largedensity}:}
Fix a graph $H$ on $n$ vertices and a function $\beta:(0,1]
\rightarrow (0,1]$. Note that if a graph $G$ of order $N$ is
$(\beta,\delta)$-dense, then, defining
$\beta_{\alpha}(\sigma)=\beta(\alpha\sigma)$ for $0<\alpha \leq 1$,
every induced subgraph of $G$ of size at least $\alpha N$ is
$(\beta_{\alpha},\alpha^{-1}\delta)$-dense. Therefore, Lemma
\ref{betalemma} implies that there is
$\delta=\delta(\beta,\epsilon,H,\alpha)$ such that every
$(\beta,\delta)$-dense $H$-free graph is
$(\alpha,\delta,\epsilon,2)$-dense.

We first show by induction on $t$ that for $\alpha,\epsilon>0$ and
positive integer $t$, there is $\delta>0$ such that every
$(\beta,\delta)$-dense $H$-free graph $G$ is
$(\alpha,\delta,\epsilon,2^t)$-dense. We have already established
the base case $t=1$. In particular, for $\alpha,\epsilon>0$ there is
$\delta'>0$ such that every $(\beta,\delta')$-dense $H$-free graph
is $(\alpha,\delta',\epsilon/4,2)$-dense. Our induction hypothesis
is that for $\alpha',\epsilon>0$ there is $\delta^*>0$ such that
every $(\beta,\delta^*)$-dense $H$-free graph $G$ is
$(\alpha',\delta^*,\epsilon,2^{t-1})$-dense. Letting
$\alpha'=\frac{1}{2}\alpha\delta'$ and
$\delta=\frac{1}{2}\delta'\delta^*$, then by the key lemma in
\cite{FoSu} mentioned above, we have that every
$(\beta,\delta)$-dense $H$-free graph is
$(\alpha,\delta,\epsilon,2^t)$-dense, which completes the induction.

If we use the last statement with $t=\log \frac{1}{\epsilon}$ and
$\alpha=1$, then we get that there are disjoint subsets
$W_1,\ldots,W_{t} \subset V$ with $t = \frac{1}{\epsilon}$,
$|W_1|=\ldots=|W_t|=\delta|V|$, and $d_{G}(W_1,\ldots,W_t) \geq
1-\epsilon$. Since ${|W_1| \choose 2} \leq \frac{\epsilon}{t}{t|W_1|
\choose 2}$, then even if there are no edges in each $W_i$, the edge
density in the set $W_1 \cup \ldots \cup W_t$ is at least
$1-2\epsilon$. Therefore, (using $\epsilon/2$ instead of $\epsilon$)
we have completed the proof of Theorem \ref{largedensity}. \qed

\vspace{0.1cm} \noindent We use Theorem \ref{largedensity} in the
next section to establish the results on disjoint edges in simple
topological graphs. For the proof of Theorem \ref{cor:drawing}, we
need to know the dependence of $\delta$ on $\beta$ in Theorem
\ref{largedensity}. Fix $\epsilon>0$ and $H$, and let
$\beta(\sigma)=\gamma\sigma$. A careful analysis of the proof of
Lemma \ref{betalemma} demonstrates that there is a constant
$c'=c'(\epsilon,H)$ such that in Lemma \ref{betalemma} we may take
$\delta=\Omega(\gamma^{c'})$. Similarly, the above proof shows that
there is a constant $c=c(\epsilon,H)$ such that in Theorem
\ref{largedensity} we may take $\delta=\Omega(\gamma^{c})$.

R\"odl's theorem was extended by Nikiforov \cite{Ni}, who showed
that if a graph has only few induced copies of $H$, then it can be
partitioned into a constant number of sets each of which is either
very sparse or very dense. We would like to remark that our proof
can be easily modified to give a similar extension of Theorem
\ref{largedensity}, which shows that a $(\beta,\delta)$-dense graph
with few induced copies of $H$ has a partition into a constant
number of very dense subsets.

\section{Edge intersection patterns in simple topological graphs}

We next provide details of the proofs of Theorem
\ref{bipartitedisjoint} and Corollary \ref{cor:drawing} on disjoint
edge patterns in simple topological graphs. We first need to
establish an analogue of the well-known Crossing Lemma which states
that every simple topological graph with $n$ vertices and $m \geq
4n$ edges contains at least $\frac{m^3}{64n^2}$ pairs of crossing
edges. Using the proof of this lemma (see, e.g., \cite{AlSp})
together with the linear upper bound on the number of edges in a
thrackle, it is straightforward to obtain a similar result for
disjoint edges in simple topological graphs. For the sake of
completeness, we sketch the proof here.

\begin{lemma}\label{disjointlemma}
Every simple topological graph $G=(V,E)$ with $n$ vertices and $m
\geq 2n$ edges has at least $\frac{m^3}{16n^2}$ pairs of disjoint
edges.
\end{lemma}
{\bf Sketch of Proof:} Let $t$ be the number of disjoint edges in
$G$. The result in \cite{CaNi} that every $n$-vertex simple
topological graph without a pair of disjoint edges has at most
$\frac{3}{2}(n-1)$ edges implies that every $n$-vertex simple
topological graph with $m$ edges has at least $m-\frac{3}{2}(n-1)
\geq m-\frac{3}{2}n$ pairs of disjoint edges. Let $G'$ be the random
induced subgraph of $G$ obtained by picking each vertex with
probability $p=2n/m \leq 1$. The expected number of vertices of $G'$
is $pn$, the expected number of edges of $G'$ is $p^2n$, and the
expected number of pairs of disjoint edges in the given embedding of
$G'$ is $p^4t$. Hence, $p^4t \geq p^2m-\frac{3}{2}pn$, or
equivalently, $t \geq p^{-2}m-\frac{3}{2}p^{-3}n=\frac{m^3}{16n^2}$,
which is the desired result.\qed

Another ingredient in the proof of Theorem \ref{bipartitedisjoint}
is a separator theorem for curves proved in \cite{FoPa}. A {\it
separator} for a graph $\Gamma=(V,E)$ is a subset $V_0 \subset V$
such that there is a partition $V=V_0 \cup V_1 \cup V_2$ with
$|V_1|,|V_2| \leq \frac{2}{3}|V|$ and no vertex in $V_1$ is adjacent
to any vertex in $V_2$. Using the well-known Lipton-Tarjan separator
theorem for planar graphs, Fox and Pach \cite{FoPa} proved that the
intersection graph of any collection of curves in the plane with $k$
crossings has a separator of size at most $C\sqrt{k}$, where $C$ is
an absolute constant. Recall that Theorem \ref{bipartitedisjoint}
says that for each $\gamma>0$ there is $\delta>0$ and $n_0$ such
that every simple topological graph $G=(V,E)$ with $n \geq n_0$
vertices and $m \geq \gamma n^2$ edges contains two disjoint edge
subsets $E_1,E_2$ each of cardinality at least $\delta n^2$ such
that every edge in $E_1$ is disjoint from every edge in $E_2$.

\vspace{0.2cm}\noindent {\bf Proof of Theorem
\ref{bipartitedisjoint}:} Define an auxiliary graph $\Gamma$ with a
vertex for each edge of the simple topological graph $G$ in which a
pair of vertices of $\Gamma$ are adjacent if and only if their
corresponding edges in $G$ are disjoint. Lemma \ref{disjointlemma}
tells us that every induced subgraph of $\Gamma$ with $\sigma m \geq
\sigma \gamma n^2 \geq 2n$ vertices has at least $\frac{(\sigma
m)^3}{16n^2}$ edges and therefore has edge density at least
$$\frac{(\sigma m)^3}{16n^2}/{\sigma m \choose 2} \geq \frac{\sigma m}{8n^2} \geq \frac{\gamma
\sigma}{8}.$$ In other words, $\Gamma$ is $(\beta,\delta)$-dense
with $\beta(\sigma)=\frac{\gamma \sigma}{8}$ and
$\delta=\frac{2}{\gamma n}$. Let $H$ be the $15$-vertex graph which
is the complement of the $1$-subdivision of $K_5$. As mentioned in
Section \ref{topsubsection}, the intersection graph of curves in the
plane does not contain the $1$-subdivision of $K_5$ as an induced
subgraph and therefore the graph $\Gamma$ is $H$-free. Hence,
Theorem \ref{largedensity} implies that for each $\epsilon>0$ there
is $\delta'>0$ and an induced subgraph $\Gamma'$ of $\Gamma$ with
order at least $\delta' m \geq \delta' \gamma n^2$ and edge density
at least $1-\epsilon$. We use this fact with
$\epsilon=\frac{1}{36C^2}$, where $C$ is the constant in the
separator theorem for curves. Since $\Gamma'$ has edge density at
least $1-\epsilon$ and each pair of edges in the simple topological
graph cross at most once, then the number $k$ of crossings between
edges of $G$ corresponding to vertices of $\Gamma'$ is less than
$\epsilon|\Gamma'|^2=\frac{1}{36C^2}|\Gamma'|^2$. Applying the
separator theorem for curves, we get a partition of the vertex set
of $\Gamma'$ into subsets $V_0,V_1,V_2$ with $|V_0| \leq
C\sqrt{\frac{1}{36C^2}|\Gamma'|^2} \leq |\Gamma'|/6$ and
$|V_1|,|V_2| \leq 2|\Gamma'|/3$, and no edges in $\Gamma'$ between
$V_1$ and $V_2$. In particular, both $V_1$ and $V_2$ have
cardinality at least $|\Gamma'|/6$. Therefore, letting
$\delta=\frac{1}{6}\delta'\gamma$, we have two edge subsets
 $E_1,E_2$ of $G$ (which correspond to $V_1,V_2$ in $\Gamma'$) each
 with cardinality at least $\delta n^2$ such that every
 edge in $E_1$ is disjoint from every edge in $E_2$. \qed

As we already mentioned in the discussion right after the proof of
Theorem \ref{largedensity}, the value of $\delta'$ which was used in
the above proof of Theorem \ref{bipartitedisjoint} satisfies
$\delta' \geq \gamma^{c'}$ for some constant $c'$. Since
$\delta=\frac{1}{6}\delta'\gamma \geq \frac{1}{6}\gamma^{c'+1}$,
 we have the following quantitative version of Theorem
\ref{bipartitedisjoint}. There is a constant $c$ such that every
simple topological graph $G=(V,E)$ with $n$ vertices and at least
$\gamma n^2$ edges with $\gamma \geq 2/n$ has two disjoint edge
subsets $E_1,E_2 \subset E$ each of size at least $\gamma^c n^2$
such that every edge in $E_1$ is disjoint from every edge in $E_2$.

We next prove a strengthening of Corollary \ref{cor:drawing}. It
says that any simple topological graph on $n$ vertices and at least
$\gamma n^2$ edges contains $\gamma'(\log n)^{1+a}$ disjoint edges
where $\gamma'>0$ only depends on $\gamma$ and $a>0$ is an absolute
constant.

\vspace{0.2cm}\noindent {\bf Proof of Corollary \ref{cor:drawing}:}
Let $d$ be the largest positive integer such that $\gamma^{c^d} \geq
n^{-1/2}$, where $c$ is the constant in the quantitative version of
Theorem \ref{bipartitedisjoint} stated above. By repeated
application of this quantitative version, we get disjoint subsets
$E_1,\ldots,E_{2^d}$ each of size at least $\gamma^{c^d}n^2 \geq
n^{3/2}$ such that no edge in $E_i$ intersects an edge in $E_j$ for
all $i \not = j$. By definition of $d$, we have $\gamma^{c^{d+1}} <
n^{-1/2}$, which implies that $2^d \geq \left(\frac{\log n}{2c\log
1/\gamma} \right)^{1/\log c}=\gamma_1(\log n)^{b}$ where
$\gamma_1>0$ only depends on $\gamma$ and $b=1/\log c>0$ is an
absolute constant. Now we need to use the result of Pach and T\'oth
\cite{PaTo} mentioned in Section \ref{topsubsection}, which says
that every simple topological graph of order $n$ without $k$
pairwise disjoint edges has $O(n (\log n)^{4k-8})$ edges. By
choosing $k'=\frac{\log n}{8\log \log n}$, we conclude that every
simple topological graph with at least $n^{3/2}$ edges (in
particular, each of the sets $E_i$) contains at least $k'$ pairwise
disjoint edges. Therefore, altogether $G$ contains $\gamma_1(\log
n)^{b} \cdot \frac{\log n}{8\log \log n} \geq \gamma'(\log n)^{1+a}$
pairwise disjoint edges, where $\gamma'>0$ only depends on $\gamma$
and $a>0$ is any absolute constant less than $b$. \qed

\section{Monochromatic Induced Copies}\label{multicolorinduced}

The goal of this section is to prove the upper bound on multicolor
induced Ramsey numbers in Theorem \ref{inducedRamsey}. To accomplish
this, we demonstrate that the graph $\Gamma$ which gives the bound
in this theorem can be taken to be any pseudo-random graph of
appropriate order and edge density. Recall that the {\it random
graph} $G(n,p)$ is the probability space of labeled graphs on $n$
vertices, where every edge appears independently with probability
$p$. An important property of $G(n,p)$ is that, with high
probability, between any two large subsets of vertices $A$ and $B$,
the edge density $d(A,B)$ is approximately $p$, where $d(A,B)$ is
the fraction of ordered pairs $(a,b) \in A \times B$ that are edges.
This observation is one of the motivations for the following useful
definition. A graph $\Gamma=(V,E)$ is {\it
$(p,\lambda)$-pseudo-random} if the following inequality holds for
all (not necessarily disjoint) subsets $A,B \subset V$:
$$|d(A,B)-p| \leq \frac{\lambda}{\sqrt{|A||B|}}.$$
The survey by Krivelevich and Sudakov \cite{KrSu} contains many
examples of $(p,\lambda)$-pseudo-random graphs on $n$ vertices with
$\lambda=O(\sqrt{pn})$. One example is the random graph $G(n,p)$
which with high probability is $(p,\lambda)$-pseudo-random with
$\lambda=O(\sqrt{pn})$ for $p<.99$. The Paley graph $P_N$ is another
example of a pseudo-random graph. For $N$ a prime power, $P_N$ has
vertex set $\mathbb{F}_N$ and distinct elements $x,y \in
\mathbb{F}_N$ are adjacent if $x-y$ is a square. It is well known
(see, e.g., \cite{KrSu}) that the Paley graph $P_N$ is
$(1/2,\sqrt{N})$-pseudo-random. We deduce Theorem
\ref{inducedRamsey} from the following theorem.

\begin{theorem}\label{inducedpseudorandom}
If $n,k \geq 2$ and $\Gamma$ is $(p,\lambda)$-pseudo-random with $N$
vertices, $0<p \leq 1/2$, and $\lambda \leq (k/p)^{-100n^3k}N$, then
every graph on $n$ vertices has a monochromatic induced copy in
every $k$-edge-coloring of $\Gamma$. Moreover, all of these
monochromatic induced copies can be found in the same color.
\end{theorem}

By letting $\Gamma$ be a sufficiently large, pseudo-random graph
with $p=1/2$, Theorem \ref{inducedRamsey} follows from Theorem
\ref{inducedpseudorandom}. For example, with high probability, the
graph $\Gamma$ can be taken to be the random graph $G(N,1/2)$ with
$N=k^{500n^3k}$. Alternatively, for an explicit construction, we can
take $\Gamma$ to be a Paley graph $P_N$ with $N \geq k^{500n^3k}$
prime.

The following lemma is the main tool in the proof of Theorem
\ref{inducedpseudorandom}. In the setting of Lemma \ref{helpful}, we
have a graph $G$ that is a subgraph of a pseudo-random graph
$\Gamma$. We use Lemma \ref{lem:dependent} to show there is a large
subset $A$ of vertices such that $|N_G(S_1)|$ is large for almost
all small subsets $S_1$ of $A$. We use the pseudo-randomness of
$\Gamma$ to ensure that for almost all small disjoint subsets $S_1$
and $S_2$ of $A$, there are many vertices adjacent to $S_1$ in $G$
and adjacent to $S_2$ in $\bar \Gamma$.

\begin{lemma}\label{helpful}
Let $\Gamma$ be a $(p,\lambda)$-pseudo-random graph with $p \leq
1/2$ and $G$ be a subgraph with order $N$ and $\epsilon{N \choose
2}$ edges. Suppose $m$ and $n$ are positive integers such that
$$8n(\lambda/p)^{\frac{2}{2n+1}}N^{1-\frac{2}{2n+1}}<m<\frac{\epsilon^2}{2^{10pn+4}n}N.$$
Then there is a subset $A \subset V$ with $|A| \geq
\frac{1}{4}\epsilon^{4n}N$ such that for all but less than
$(2n)^{-2n}{m \choose n}^2$ pairs of disjoint subsets $S_1,S_2
\subset A$ with $|S_1|=|S_2|=n$, there are at least $m$ vertices
adjacent to every vertex of $S_1$ in $G$ and no vertex of $S_2$ in
$\Gamma$.
\end{lemma}
\begin{proof}
By averaging over all partitions $V=V_1 \cup V_2$ of $G$ with
$|V_1|=|V_2|=N/2$, we can find a partition with at least $\epsilon
(N/2)^2$ edges between $V_1$ and $V_2$. By Lemma \ref{lem:dependent}
with $a=1$, $t=4n$, $d=n$, and $x=(1-3p/2)^{-n}m$, there is a subset
$A \subset V_1$ with cardinality at least
$\frac{1}{2}\epsilon^{4n}|V_1| =\frac{1}{4}\epsilon^{4n}N$ such that
the number of subsets $S_1$ of $A$ of size $n$ with $|N_G(S_1)| < x$
is at most
\begin{equation*}
2\epsilon^{-4n}\left(\frac{x}{N/2}\right)^{4n}\left(\frac{|A|}{N/2}\right){N/2
\choose n} \leq \left(\frac{2x}{\epsilon N}\right)^{4n}{N \choose n}
\leq \epsilon^{-4n}2^{20pn^2+4n}(m/N)^{4n}{N \choose n},
\end{equation*}
where the last inequality follows from the simple inequality $1-3p/2
\geq 2^{-5p}$ for $p \leq 1/2$. This implies, using the upper bound
on $m$, that the number of disjoint pairs $S_1,S_2$ of subsets of
$A$ with $|S_1|=|S_2|=n$ and $|N_G(S_1)|<x$ is at most
\begin{eqnarray}\label{inducedrams2}
\epsilon^{-4n}2^{20pn^2+4n}(m/N)^{4n}{N \choose n} \cdot {N \choose
n} &\leq&
\epsilon^{-4n}2^{20pn^2+4n}(m/N)^{2n}\frac{m^{2n}}{n!^2} \nonumber\\
&\leq&
\epsilon^{-4n}2^{20pn^2+4n}\left(\frac{\epsilon^2}{2^{10pn+4}n}\right)^{2n}\frac{m^{2n}}{n!^2}
\leq 2^{-4n}n^{-2n}\frac{m^{2n}}{n!^2} \nonumber \\ &<&
\frac{1}{2}(2n)^{-2n}{m \choose n}^2.
\end{eqnarray}

Let $S_1$ be a subset of $A$ of cardinality $n$ with $|N_G(S_1)|
\geq x=(1-3p/2)^{-n} m$. We will show that almost all subsets $S_2$
of $A \setminus S_1$ of cardinality $n$ satisfy $|N_G(S_1) \cap
N_{\bar \Gamma}(S_2)| \geq m$. To do this, we give a lower bound on
the number of ordered $n$-tuples $(u_1,\ldots,u_n)$ of distinct
vertices of $A\setminus S_1$ such that for each $i$, $|N_G(S_1) \cap
N_{\bar \Gamma}(\{u_1,\ldots,u_i\})| \geq (1-3p/2)^i|N_G(S_1)|$.
Suppose we have already picked $u_1,\ldots,u_{i-1}$ satisfying
$|N_G(S_1) \cap N_{\bar \Gamma}(\{u_1,\ldots,u_{i-1}\})| \geq
(1-3p/2)^{i-1}|N_G(S_1)| \geq m$. Let $X_i$ denote the set of
vertices $u_i$ in $A$ with $|N_G(S_1) \cap N_{\bar \Gamma}
(\{u_1,\ldots,u_i\})| < (1-3p/2)^i|N_G(S_1)|$. Then the edge density
between $X_i$ and $N_G(S_1) \cap N_{\bar \Gamma}
(\{u_1,\ldots,u_{i-1}\})$ in $\Gamma$ is more than $3p/2$. Since
$\Gamma$ is $(p,\lambda)$-pseudo-random, we have $$p/2 <
\frac{\lambda}{\sqrt{|X_i|\cdot |N_G(S_1) \cap N_{\bar \Gamma}
(\{u_1,\ldots,u_{i-1}\}|)}} \leq \frac{\lambda}{\sqrt{|X_i|m}}.$$
Therefore, $|X_i| <4(\lambda/p)^2m^{-1}$. Hence, during this process
for every index $i$ there are at least
$$|A|-|S_1|-(i-1)-\frac{4(\lambda/p)^2}{m} >|A|-2n-4(\lambda/p)^2m^{-1} \geq |A|-8n(\lambda/p)^2m^{-1}$$
choices for $u_i \in A \setminus (S_1 \cup \{u_1,\ldots,u_{i-1}\}
\cup X_i)$. Therefore, given $S_1$ with $|N_G(S_1)| \geq x$, we
conclude that the number of ordered $n$-tuples $(u_1,\ldots,u_n)$ of
distinct vertices of $A \setminus S_1$ with
$$|N_G(S_1) \cap N_{\bar \Gamma}(\{u_1,\ldots,u_n\})| \geq
(1-3p/2)^n|N_G(S_1)| \geq m$$ is at least
$$\left(|A|-8n(\lambda/p)^2m^{-1}\right)^n \geq |A|^n-8n^2(\lambda/p)^2m^{-1}|A|^{n-1}.$$
This together with the lower bound on $m$ implies that the number of
pairs $S_1,S_2$ of disjoint (unordered) subsets of $A$ with
$|N_G(S_1)| \geq x$ and $|N_G(S_1) \cap N_{\bar \Gamma}(S_2)| <m$ is
at most
\begin{eqnarray}\label{inducedrams3}
 {N \choose n} \cdot
\frac{1}{n!}8n^2(\lambda/p)^2m^{-1}|A|^{n-1} &\leq&
8n^2m^{-1}n!^{-2}(\lambda/p)^2 N^{2n-1} \leq 8n^2m^{-1}n!^{-2}
\left(\frac{m}{8n}\right)^{2n+1} \nonumber\\ &=& 2^{-4n}n
(2n)^{-2n}\frac{m^{2n}}{n!^2} < \frac{1}{2}(2n)^{-2n}{m \choose
n}^2.
\end{eqnarray}

Combining (\ref{inducedrams2}) and (\ref{inducedrams3}), all but
less than $(2n)^{-2n}{m \choose n}^2$ pairs $S_1,S_2$ of disjoint
subsets of $A$ with $|S_1|=|S_2|=n$ satisfy $|N_G(S_1) \cap N_{\bar
\Gamma}(S_2)| \geq m$, which completes the proof.
\end{proof}

We are now ready to prove our main result in this section.

\vspace{0.2cm} \noindent {\bf Proof of Theorem
\ref{inducedpseudorandom}:} Consider a $k$-edge-coloring of the
$(p,\lambda)$-pseudo-random graph $\Gamma$ with colors $1,\ldots,k$.
Let $B_1$ denote the set of vertices of $\Gamma$. For $j \in
\{1,\ldots,k\}$, let $G_j$ denote the graph of color $j$. Let
$\epsilon=\frac{p}{2k}$ and $m= \epsilon^{20n^2k}N$.

We will pick nested subsets $B_1 \supset \ldots \supset
B_{k(n-2)+2}$ such that, for each $i \leq k(n-2)+1$, we have
$|B_{i+1}| \geq \frac{1}{4}\epsilon^{4n}|B_{i}|$ and there is a
color $c(i) \in \{1,\ldots,k\}$ such that all but less than
$(2n)^{-2n}{m \choose n}^2$ pairs of disjoint subsets $S_1,S_2
\subset B_{i+1}$ each of size $n$ have at least $m$ vertices in
$B_i$ adjacent to $S_1$ in $G_{c(i)}$ and adjacent to $S_2$ in $\bar
\Gamma$. Once we have found such a family of nested subsets, the
proof is easy. By the pigeonhole principle, one of the $k$ colors is
represented at least $n-1$ times in the sequence
$c(1),\ldots,c(k(n-2)+1)$. We suppose without loss of generality
that $1$ is this popular color. Let $i(1)=1$ and for $1<j \leq n-1$,
let $i(j)$ be the $j^{\textrm{th}}$ smallest integer $i>1$ such that
$c(i-1)=1$. Letting $A_j=B_{i(j)}$, we have, by Lemma
\ref{inducedembedding} with $G_1$ as $G$ and $\bar \Gamma$ as $F$,
that there is an induced copy of every graph on $n$ vertices that is
monochromatic of color $1$. So, for the rest of the proof, we only
need to show that there are nested subsets $B_1 \supset \ldots
\supset B_{k(n-2)+2}$ and colors $c(1),\ldots,c(k(n-2)+1)$ with the
desired properties.

We now show how to pick $c(i)$ and $B_{i+1}$ having already picked
$B_i$. Let $c(i)$ denote the densest of the $k$ colors in
$\Gamma[B_i]$. By pseudo-randomness of $\Gamma$, it is
straightforward to check that the density of $\Gamma$ in $B_i$ is at
least $p/2$, so the edge density of color $c(i)$ in $G[B_i]$ is at
least $\frac{p}{2k}=\epsilon$. Indeed, if not, then the density
between $B_1$ and itself in $\Gamma$ deviates from $p$ by at least
$p/2$ and so, by pseudo-randomness of $\Gamma$,
$$\frac{2\lambda}{p} \geq |B_i| \geq \left(\frac{1}{4}\epsilon^{4n}\right)^iN \geq
\left(\frac{p}{4k}\right)^{4n^2k}N \geq (p/k)^{20n^2k}N,$$
contradicting the upper bound on $\lambda$. Since
$m=\epsilon^{20n^2k}N$, $k \geq 2$, $p \leq 1/2$, and
$\epsilon=\frac{p}{2k} \leq 1/8$, we have
\begin{eqnarray*}
8n(\lambda/p)^{\frac{2}{2n+1}}|B_i|^{1-\frac{2}{2n+1}}&\leq&
8n\left((k/p)^{-100n^3k}N
 /p\right)^{\frac{2}{2n+1}}N^{1-\frac{2}{2n+1}} <
 8n(k/p)^{-50n^2k}N \\ &<& m <
 \frac{\epsilon^2}{2^{20n}}\left(\frac{1}{4}\epsilon^{4n}\right)^{nk} N<
 \frac{\epsilon^2}{2^{10pn+4}n}|B_i|.\end{eqnarray*}
Hence, we may apply Lemma \ref{helpful} to the graph
$G_{c(i)}[B_i]$, which is a subgraph of the
$(p,\lambda)$-pseudo-random graph $\Gamma$, and get a subset
$B_{i+1}$ of $B_i$ with the desired properties. This completes the
proof of the induction step and the proof of theorem. \qed

\section{Concluding Remarks}

\begin{itemize}
\item We conjecture that there is an absolute constant $c$ such that $r(H) \leq 2^{c\Delta}n$ for every
$H$ with $n$ vertices and maximum degree $\Delta$ and our results
confirm it for graphs of bounded chromatic number. This question is
closely related to another old problem on Ramsey numbers. More than
thirty years ago, Erd\H{o}s conjectured that $r(H) \leq
2^{c\sqrt{m}}$ for every graph $H$ with $m$ edges and no isolated
vertices. The best known bound for this question is $r(H) \leq
2^{c\sqrt{m}\log m}$ (see \cite{AlKrSu}) and the solution of our
conjecture might lead to further progress on the problem of
Erd\H{o}s as well.

\item The bound $r(H) \leq  2^{4d+12}\Delta n$ for bipartite $d$-degenerate $n$-vertex graphs
with maximum degree $\Delta$ shows that $\log r(H) \leq 4d+2\log n +
12$. On the other hand, the standard probabilistic argument gives
the lower bound $r(H) \geq \max(2^{d(H)/2},n)$, where the {\it
degeneracy number} $d(H)$ is the smallest $d$ such that $H$ is
$d$-degenerate. It therefore follows that $\log r(H)
=\Theta\left(d(H)+\log n\right)$ for every bipartite graph $H$ (this
can be also deduced with slightly weaker constants from the result
in \cite{AlKrSu}). In particular, since $d(H)$ can be quickly
computed by simply deleting the vertex with minimum degree and
repeating this until the graph is empty, we can efficiently compute
$\log r(H)$ up to a constant factor for every bipartite graph $H$.
It is plausible that $\log r(H)=\Theta(d(H)+\log n)$ for every
$d$-degenerate $n$-vertex graph $H$. If so, then this would imply
the above mentioned conjecture of Erd\H{o}s that $r(H) \leq
2^{c\sqrt{m}}$ for every graph $H$ with m edges and no isolated
vertices since every such graph satisfies $d(H)=O(\sqrt{m})$.

\item The exciting conjecture of Burr and Erd\H{o}s,
which was the driving force behind most of the research done on
Ramsey numbers for sparse graphs, is still open. Moreover, we even
do not know how to deal with the interesting special case of
bipartite graphs in which all vertices in one part have bounded
degree. However, the techniques in this paper can be used to make
modest progress and solve this special case when the bipartite graph
$H$ is bi-regular, i.e., every vertex in one part has degree
$\Delta_1$ and every vertex in the other part has degree $\Delta_2$.
The proof of the following theorem is a minor variation of the proof
of Theorem \ref{main} and therefore is omitted.

\begin{theorem}
Let $H=(V_1,V_2)$ be a bipartite graph without isolated vertices
such that, for $i \in \{1,2\}$, the number of vertices in $V_i$ is
$n_i$ and the maximum degree of a vertex in $V_i$ is $\Delta_i$.
Then $r(H) \leq 2^{c\Delta_1}\Delta_2n_2$ for some absolute constant
$c$.
\end{theorem}

Note that this theorem implies that if $H$ also satisfies
$\Delta_2n_2=2^{O(\Delta_1)}n$, then $r(H) =2^{O(\Delta_1)}n$, where
$n$ is the number of vertices of $H$. In particular, this bound is
valid for graphs whose average degree in each part is at least a
constant fraction of the maximum degree in that part.

Also, it is possible to extend ideas used in the proofs of Theorems \ref{degenerate}
and \ref{chromatic} to show that for every $0<\delta\leq 1$ the Ramsey number of any
$d$-degenerate graph $H$
with $n$ vertices and maximum degree $\Delta$ satisfies
$r(H) \leq 2^{c/\delta} \Delta^{\delta}n$, where $c$ is a constant depending
only on $d$. By taking $\delta=(\log n)^{-1/2}$ we have that
$r(H) \leq 2^{c(d)\sqrt{\log n}} n$ for every $d$-degenerate graph of order $n$.
This improves the result in \cite{KoSu}.

\item One should be able to extend the bound in Theorem \ref{bipartitedisjoint}
to work for all possible sizes of simple topological graphs.
Moreover, it might be true that every simple topological graph with
$m=\epsilon n^2$ edges with $\epsilon \geq 2/n$ contains two sets of
size $\delta n^2$ of pairwise disjoint edges with
$\delta=c\epsilon^2$ for some absolute constant $c>0$. This would
give both Theorem \ref{bipartitedisjoint} and, taking
$\epsilon=2/n$, a linear bound on the size of thrackles. For
comparison, our proof of Theorem \ref{bipartitedisjoint}
demonstrates that $\delta$ can be taken to be a polynomial in
$\epsilon$.

It would be also interesting to extend Conway's conjecture by
showing that for every fixed $k$, the number of edges in a simple
topological graph with $n$ vertices and no $k$ pairwise disjoint
edges is still linear in $n$. This is open even for $k=3$, though
(see Section \ref{topsubsection}) an almost linear upper bound was
given in \cite{PaTo}. For geometric graphs, such a linear bound was
a longstanding conjecture of Erd\H{o}s and Perles and was settled in
the affirmative by Pach and T\"orocsik.

\end{itemize}

\vspace{0.1cm}
\noindent
{\bf Acknowledgment.}\,
We'd like to thank Steve Butler
for carefully reading this manuscript.

\vspace{0.2cm}
\noindent
{\bf Note added in proof.}\, After this paper was written we learned that
D. Conlon proved the following variant of Corollary \ref{cormain},
independently and simultaneously with our work.
He showed that $r(H) \leq 2^{(2+o(1))\Delta} n$
for bipartite $n$-vertex graph $H$ with maximum degree $\Delta$.

\end{document}